\documentclass[oneside, reqno, 11pt]{amsart}
\usepackage{amssymb,amscd, amsthm,amsmath}
\usepackage{comment}
\usepackage{layout}
\usepackage[colorlinks=true,urlcolor=blue,linkcolor=blue]{hyperref}
\usepackage{graphicx}
\usepackage{tikz}
\usepackage[centerlast,small]{caption}
\allowdisplaybreaks

\newtheorem{theorem}{Theorem}[section]
\newtheorem{corollary}[theorem]{Corollary}
\newtheorem{lemma}[theorem]{Lemma}

\newtheorem{proposition}[theorem]{Proposition}
\newtheorem{example}[theorem]{Example}

\newtheorem{remark}[theorem]{Remark}

\def\be{\beta'_t}

\def\ve{\varepsilon}

\def\compl{\mathbb{C}}

\def\nat{\mathbb{N}}

\def\real{\mathbb{R}}

\def\q{\quad}

\def\mes{\mathcal{M}}

\def\conv{\boxplus}

\def\unve{\underline\varepsilon}
\def\mutilda{\omega \boxplus \mu_s^{(\underline\varepsilon_s)}}
\def\muves {\tilde\mu_{m+s}}
\def\muvesthree {\tilde\mu_{m + s}}

\def\set {E_{m,s}^n}

\def \uns {\underline \sigma}
\def \est {\delta^{-6}}

\def\bee{\begin{eqnarray}}
\def\eee{\end{eqnarray}}

\def\be*{\begin{eqnarray*}}
\def\ee*{\end{eqnarray*}}

\numberwithin{equation}{section}

\begin{document}

\title {ASYMPTOTIC EXPANSIONS IN FREE LIMIT THEOREMS}

\author{F. G\"OTZE$^1$ }
\thanks{$^1$Reserach supported by CRC 701.}
\author{
A. RESHETENKO$^2$}
\thanks{$^2$Research supported by IRTG 1132.} 

\begin{abstract}
We study asymptotic expansions in  free probability. 
In a class of classical limit theorems  Edgeworth expansion
can be obtained via 
a general approach using sequences of ``influence'' functions of individual random elements described by
vectors of real parameters $(\ve_1,\dots, \ve_n)$, that is by
a sequence of functions 
$h_n(\varepsilon_1,\dots, \varepsilon_n;t)$, $|\varepsilon_j| \leq \frac 1 {\sqrt n}$, $j=1,\dots,n$, 
$t\in A \subset \real$ (or $\compl$) 
which are smooth, symmetric, compatible and have vanishing first derivatives at zero.
In this work we expand this approach to free probability.
As  a sequence of functions $h_n(\varepsilon_1,\dots, \varepsilon_n;t)$ we consider 
a sequence of the Cauchy transforms of the sum
$\sum_{j=1}^n \varepsilon_j X_j $, where $(X_j)_{j=1}^n$ are
free identically distributed random  variables with nine moments. We derive   Edgeworth type expansions for distributions and densities (under the additional assumption that $\mbox{supp} (X_1) \subset [-\sqrt [3]n, \sqrt [3]n]$) of the  sum 
$\frac 1 {\sqrt n} \sum_{j=1}^n X_j$
within the interval $(-2,2)$. 
\end{abstract}

\maketitle


\section {Introduction}

Free probability theory was initiated by  Voiculescu in 1980's as a tool for understanding free group factors.
The main concept in this theory  is  the notion of freeness, which is a counterpart of the classical independence for non-commutative random variables.

The distribution of the sum of two free random variables is uniquely determined by the distributions of the summands and called the free  convolution of the initial distributions.
While classical convolutions are studied via Fourier transforms, free convolutions can be studied via Cauchy transforms.
Numerous results concerning the distributional behaviour of the sum of several free random variables were proved in the recent years: 
Free limit theorems \cite {M92, V86}, 
the law of large numbers \cite {BP96}, 
the Berry-Esseen inequality \cite {CHG08, Kar07},
the Edgeworth expansion in the free central limit theorem \cite {ChG11v2} etc.
These results parallel  the classical ones. 
On the other hand some results in free probability theory have no counterparts in classical probability theory.
For example, the so called superconvergence. 
This type of convergence appears in free limit theorems and is stronger then usual convergence.

The Edgeworth type expansion in free probability theory was first obtained by Chistyakov and G\"otze in \cite{ChG11v2}.
The idea is based on the approximation of the distribution of $Y_n: = \frac 1 {\sqrt n} \sum_{j=1}^n X_j$, where $X_j$, $j=1,\dots, n$ are
 free identically distributed random variables.  by the shifted free Meixner distribution,
The expansion for the distribution and density of $Y_n$ is given at the point $x+ m_3/\sqrt n$, where $m_3$ is the third moment of $X_1$. 

In this paper we develop a technique which was described in \cite{G85}. 
This approach (see Section 4) was introduced as a tool to derive asymptotic expansions and estimates for the reminder term  in a class of classical functional limit theorems in abstract spaces. 
It is based on  Taylor expansions only  and hence can be applied in free probability  without additional modifications.
We use this method and derive the Edgeworth expansions for distributions and densities of normalized sums
$Y_n$.

The paper is organized as follows. In Section 2 we formulate and discuss the main results. 
Preliminaries are introduced in Section 3. In Section 4 we describe the general scheme. 
In Section 5 we apply this general scheme to free probability.
Section 6 is devoted to the proofs of results. 
In the Appendix we provide formulations of some results of the  literature, in particular a more detailed and revised version of the expansion scheme outlined in \cite{G85} for the readers convenience.
The results of this paper are part of the Ph.D. thesis of the second author in 2014 at the University of Bielefeld.


\section{Results}

Denote by $\mes$ the family of all Borel probability measures defined on the real line $\real$.

Let $X_1, X_2, \dots$ be free self-adjoint identically distributed random variables with distribution 
$\mu \in \mes$.
Denote by $m_k$ and $\beta_k$ the moments and absolute moments of $\mu$.
Throughout the text we
assume that   $\mu$ has zero mean and unit variance.
Let $\mu_n$ be the distribution of the normalized sum 
$\frac 1 {\sqrt n} \sum _{j=1}^n X_j$.
In free probability a sequence of measures $\mu_n$
 converges to the semicircle law $\omega$ as $n$ tends to $\infty$. Moreover, $\mu_n$  is absolutely continuous with respect to the Lebesgue measure for sufficiently large $n$ \cite{W10}.

We denote by $p_{\mu_n}$ the density of $\mu_n$.
Define the Cauchy transform of a measure $\mu$:
$$G_\mu(z) = \int_{\real} \frac {\mu(dx)}{z - x},\ \ z\in \compl^+,$$
where $\compl^+$ denotes the upper half plane.


In \cite{ChG11v2} Chistyakov and G\"otze  obtained  
a formal power expansion for the Cauchy transform of $\mu_n$
and  the Edgeworth type expansions for $\mu_n$ and $p_{\mu_n}$.
Below we review these results.
Assume that $\mu$  has compact support. 
Denote by $U_n(x)$ the Chebyshev polynomial of the second kind of degree $n$, which is given by the recurrence relation:
\begin{eqnarray}
\label{Cheb}
U_0(x) = 1,\ \ 
U_1(x)  = 2x,\ \ 
U_{n+1}(x)  = 2x U_n(x) - U_{n-1}(x).
\end{eqnarray}
The formal expansion has the form
\begin{eqnarray}
\label{form_c}
G_{\mu_n}(z) = G_{\omega}(z) + \sum_{k = 1}^\infty \frac {B_k(G_{\omega}(z))}{n^{k/2}},
\end{eqnarray}
where 
\begin{eqnarray}
\label{coef}
B_k(z) = \sum_{(p,m)} c_{p,m}\frac {z^p}{(1/z - z)^m}
\end{eqnarray}
with real coefficients $c_{p,m}$ which depend on the free cumulants $\kappa_3, \dots, \kappa_{k+2}$ 
and do not depend on $n$. The free cumulants will be defined in Section 2.
The  summation on the right-hand side of (\ref{coef}) is taken over a finite set of non-negative integer pairs
$(p,m)$. 
The coefficients $c_{p,m}$ can be calculated explicitly. For the cases $k=1,2$ we have
$$
B_1(z) = \frac { \kappa_3 z^3} {1/z - z},\ \ 
B_2(z) = \frac {(\kappa_4 - \kappa_3^2) z^4}{1/z - z} + 
\kappa_3^2\left( \frac {z^5}{(1/z - z)^2} + \frac {z^2}{(1/z - z)^3}\right).
$$

Let us introduce some further notations. 
Denote by $\beta_q$ the  $q$th absolute moment of $\mu$, and 
assume that $\beta_q < \infty$ for some $q \geq 2$. 
Moreover, denote
\begin{eqnarray*}
a_n: = \frac {\kappa_3} {\sqrt n},\q b_n: = \frac {\kappa_4 - \kappa_3^2 + 1}n,\q 
d_n: = \frac {\kappa_4 - \kappa_3^2 + 2}n,\q n \in \nat.
\end{eqnarray*}
Introduce the Lyapunov fractions
\begin{eqnarray*}
L_{qn}: = \frac {\beta_q}{n^{(q-2)/2}} \q  \mbox{and let} \q
\rho_q(\mu_t): = \int_{|x|>t}|x|^q  \mu(dx),\ \ t>0.
\end{eqnarray*}
Denote
$$
q_1: = \min \{q,3\},\ \   q_2: = \min \{q,4\},\ \   q_3: = \min\{q,5\}.$$
For $n \in \nat$, set
$$
\eta_{qs}(n): = \inf_{0<\ve \leq 10^{-1/2}} g_{qns}(\ve),\ \  \mbox{where}\ \  
g_{qns}(\ve): = \ve^{s + 2 -q_s}+ \frac {\rho_{q_s}(\mu, \ve \sqrt n)}{\beta_{q_s}}\ve^{-q_s}
$$
provided that $\beta_q<\infty$, $q \geq s+1$, for $s = 1,2,3$, respectively.
It is easy to see that $0 < \eta_{qs}(n) \leq 10^{1+s/2} + 1$ for $s+1 \leq q_s \leq s+2$ and
$\eta_{qs}(n) \to \infty$ monotonically as $n\to \infty$ if $s+1 \leq q_s < s+2$, and $\eta_{qs}(n) \geq 1$,
$n \in \nat$, if $q_s = s+2$.

By agreement the symbols $c,\ c_1,\ c_2,\dots$, $c(\mu),\ c_1(\mu),\ c_2(\mu),\dots$ and 
$c(\mu,s)$,$\ c_1(\mu,s)$,
$\ c_2(\mu,s),\dots$ shall denote absolute positive constants, absolute positive constants depending on $\mu$ and absolute positive constants depending on $\mu$ and $s$ respectively.

In the expansion below we do not assume  the measure $\mu$ to be of compact support. 
The distribution function  $\mu_n(-\infty, x + a_n) $ admits the expansion:

\begin{eqnarray}
\label{meas1}
\lefteqn{
\mu_n(-\infty, x + a_n)  = \omega(-\infty,x)}\q \q \\
&+& \left(\frac {a_n^2}{2}U_1\left (\frac x 2\right) 
+ 
\frac {a_n} 3\left(3 - U_2 \left( \frac x 2\right)\right) - \frac {b_n - a_n^2 -1/n}{4}
U_3\left(\frac x 2\right)\right)\! p_{\omega}(x) + \rho_{2n}(x) \nonumber
\end{eqnarray}
 for $x\in \real$, $n \in \nat$, where
\begin{eqnarray}
\label{bound_bound}
|\rho_{n2}(x)| \leq c 
\left \{
\begin{array}{rcl}
\!\!\eta_{q3}(n)L_{qn}+L_{4n}^{3/2}, &  4 \leq q < 5\\
\!\!L_{5n},\q\q\q\q\q\ \ &q  \geq 5.\q\ 
\end{array}
\right.
\end{eqnarray}

Assume that $\mu$ has compact support, then for $n \geq c_1(\mu)$, $p_{\mu_n}$ admits the  expansion 
\begin{eqnarray}
\label{exp_d}
p_{\mu_n}(x + a_n)& =& \left(1 + \frac {d_n}2 - a_n^2 - \frac 1 n - a_n x - \left(b_n - a_n^2 - \frac 1 n\right) x^2 \right)
p_{\omega}(E_n x) \nonumber\\
& \  & +\ \ \frac {c \theta}{n^{3/2} \sqrt {4 - (E_n x)^2}}
\end{eqnarray}
for $x \in [-2/E_n + h, 2/E_n - h]$, where $E_n: = (1 - b_n)/\sqrt {1 - d_n}$ and $h = \frac {c_2(\mu)}{n^{3/2}}$ and $|\theta| \leq 1$.

We formulate Edgeworth type expansions 
obtained by the general technique which is introduced in Section 3. 

First,  introduce for every $\delta \in(0,1/10)$ a rectangle $K$:
$$K: = \{x+iy: x \in [-2 +  2\delta,\  2 - 2 \delta],\ |y| < \delta\sqrt \delta\}.$$


The following corollary follows from Theorem \ref{extension_33}.

\begin{corollary}
\label{extension_g}
Assume  that $\mu \in \mes$ is supported  on $[-\sqrt [3]n, \sqrt [3] n]$ and $\beta_9 < \infty$.
For every $\delta \in (0,1/10)$ and $n$ such that $n \geq c(\mu) \est$, 
the Cauchy transform  $G_{\mu_n}$ has the  analytic extension
\begin{eqnarray*}
\label{extension1}
G_{\mu_n}(z) = G_\omega (z) + l_n(z),
\quad z \in K,
\end{eqnarray*} 
where
$|l_n(z)| \leq \frac {c}{\sqrt{\delta n}}$ on  $K$.
\end{corollary}

\begin{theorem}
\label{exp_Cauchy}
Assume  that $\mu \in \mes$ is supported  on $[-\sqrt [3]n, \sqrt [3] n]$ and $\beta_9 < \infty$.
For every $\delta \in (0,1/10)$
the extension of
the Cauchy transform $G_{\mu_n}$ admits the  expansion
\begin{eqnarray*}
\lefteqn{
G_{\mu_n}(z)  =  G_\omega (z)+ \frac {\kappa_3G_\omega ^4(z)} {(1 - G_\omega ^2(z) )\sqrt n}}\\
&+& \left(\big(\kappa_4 - \kappa_3^2\big)\frac {G^5_\omega (z)}{1 - G_\omega ^2(z)} + \kappa_3^2\left(\frac{G^7_\omega (z)}{(1 - G^2_\omega (z))^2} + \frac{G^5_\omega (z)}{(1 - G^2_\omega (z))^3}\right)\right) \frac 1 n  \\
& + &\left(\frac{\kappa_5   G_\omega ^6 (z)}{( 1 - G_\omega ^2 (z))}
- \frac{\kappa_3 \kappa_4 G^8_\omega (z) \left(5 G_\omega ^2(z) - 7 \right)}{\left(1 - G_\omega ^2(z) \right)^3}
\right.\\
& + & \left. 
 \frac{\kappa_3^3 G_\omega ^{10}(z) \left(5 G_\omega ^4(z)  - 15 G_\omega ^2(z) + 12\right)}
{\left(1 - G_\omega ^2 (z)\right)^5}
\right)\frac 1 {n^{3/2}}
+ O\left(\frac 1 {n^{2} }\right)
\end{eqnarray*} 
for $z \in K$, $n \geq c(\mu) \est$.
\end{theorem}

Due to  the Stieltjes inversion formula we  obtain an expansion for the densities.
\begin{corollary}
\label{cor5}
Assume  that $\mu \in \mes$ is supported on $[-\sqrt [3]n, \sqrt [3] n]$ and $\beta_9 < \infty$.
For every $\delta \in (0,1/10)$
the density $p_{\mu_n}$ admits the  expansion 
\begin{eqnarray*}
p_{\mu_n}(x)  =   p_\omega(x) &+& \frac{\kappa_3   \left(x^2 - 3\right) x p_\omega(x)}{(4 - x^2)\sqrt{n}}\\
& - &
 \frac{\left( \kappa_4 \left(x^6 - 8 x^4 +18 x^2 - 8\right)
- \kappa_3^2 \left(2 x^6 - 15 x^4 + 30 x^2 - 10\right) \right)p_\omega(x)}{ \left(4-x^2\right)^{2}n}\\
& + & \left( \frac{ \kappa_5  ( x^4  - 5 x^2 + 5)x}{(4 - x^2)} +
\frac{\kappa_3 \kappa_4(5 x^6  - 42 x^4  + 105 x^2 - 70 )x}{(4 - x^2)^{2}}
\right. \\
& + & \left. \frac{\kappa_3^3 (5 x^8 -60 x^6 + 252 x^4 -420 x^2 + 210)x} { (4 - x^2)^{3}}
\right) \frac {p_\omega(x)} { n^{3/2}}
  +  O\left(\frac 1 {n^{2} }\right)
\end{eqnarray*}
for $x \in [- 2 + 2 \delta, 2 - 2  \delta]$, $n \geq c(\mu) \est$.
\end{corollary}

Denote by $U_n(x)$ the Chebyshev polynomial of the second kind of degree $n$, which is given by the recurrence relation:
\begin{eqnarray}
\label{Cheb1}
U_0(x) = 1,\ \ 
U_1(x)  = 2x,\ \ 
U_{n+1}(x)  = 2x U_n(x) - U_{n-1}(x).
\end{eqnarray}

\begin{corollary}
\label{cor6}
Assume  that $\mu \in \mes$ with $\beta_9 < \infty$.
For every $\delta \in (0,1/10)$
the distribution $\mu_n$ admits the  expansion

\begin{eqnarray}
\label{distr}
\mu_n(a,b) =  
\omega(a, b) &+& \left[  -  \kappa_3U_2\left(\frac x  2\right)\frac{p_\omega(x)}{3 \sqrt{n}}\right.\\
& + & \left(-\kappa_4  U_3\left(\frac x  2\right) + 
2\kappa_3^2 \left (U_3\left(\frac x  2\right) + U_1\left(\frac x  2\right) -
 \frac {U_1\left(\frac x  2\right)} {4 -  x^2}\right)  \right) \frac{p_\omega(x)}{4n}\nonumber\\
& + &  \left(  \frac {\kappa_5} 5 U_4\left (\frac x 2\right) 
- \frac {\kappa_3 \kappa_4}{4 - x^2}\left ( U_6\left (\frac x 2\right)  - U_4\left (\frac x 2\right)\right) \right. \nonumber\\
& - & \left. \left. \frac {\kappa^3_3}{3(4 - x^2)^2} \left( 3 U_8\left (\frac x 2\right) 
 - 7 U_6\left (\frac x 2\right)  + 4 U_4\left (\frac x 2\right) \right)\right)
\frac {p_\omega(x)} { n^{3/2}}\right] \Bigg |_a^b
+ O\left(\frac 1 {n^{2} }\right) \nonumber
\end{eqnarray}
with $(a,b) \subset [- 2 + 2 \delta, 2 - 2 \delta]$, $n \geq c(\mu) \est$ and
$U_n(x)$ are Chebychev polynomials (\ref{Cheb1}).
\end{corollary}

\begin{remark}
If we assume that $\beta_k < \infty$, $k>9$, then
the above results with accuracy $O(n^{-2})$ in   Theorem \ref {exp_Cauchy}, Corollary \ref {cor5} and Corollary \ref {cor6} can be easily expanded to the higher orders using more terms of the
scheme for asymptotic expansions (\ref {expansiondf}), provided in the interval $[-2 + 2\delta, 2 - 2 \delta]$.
\end{remark}

\begin{remark}
Assume that $m_3 = 0$, then due to (\ref {distr})  we get
$$
\mu_n(a,b)  =
\omega(a,b)  -  \left[\frac {\kappa_4} {4n} U_3\left (\frac x 2\right) 
 - \frac{\kappa_5}{5 n^{3/2}}U_4\left (\frac x 2\right) \right]p_\omega(x)\Bigg |_a^b 
 +  O\left(\frac 1 { n^{2} }\right)
$$
with $(a,b) \subset [-2 + 2 \delta, 2 - 2 \delta]$, $n \geq c(\mu) \est$.
\end{remark}

In the example below we consider  asymptotic expansions for free convolutions of the free Poisson law.
\begin{example}[Free Poisson law]
{\rm
Let us consider the free Poisson law with density 
$$
p_{\mu}(x) = \frac 1 {2 \pi (x+1)} \sqrt{4(x+1) - (x+1)^2},\q  - 1 \leq x \leq 3,
$$ 
which has moments $m_1 =0$, $m_2 = 1,$ $m_3 = 1$,
$m_4 = 3$, $m_5= 6$. 
The density of $p_{\mu_n}(x)$ is given by
$$
p_{\mu_{n}}(x) = \frac{\sqrt{(4n - 1)+2 \sqrt{n} x - n x^2}}{2\pi  \left(\sqrt{n}+x\right)},\q
-2 + n^{-1/2} \leq x \leq 2 - n^{-1/2}.
$$
We consider $p_{\mu_{10}}(x)$ and $p_{\mu_{100}}(x)$:
$$
p_{\mu_{10}}(x) = \frac{\sqrt{39+2 \sqrt{10} x-10 x^2}}{2\pi  \left(\sqrt{10}+x\right)}, \q
-2 + 1/\sqrt{10} \leq x \leq 2 + 1/\sqrt{10};
$$
$$
p_{\mu_{100}}(x) = \frac{\sqrt{399+20 x-100 x^2}}{2 \pi  (10 + x)},\q
- 2 +1/10 \leq x \leq 2 - 1/10.
$$
In  Figure \ref{fig1}, one can see plots of the densities and the approximations of the densities based on 
Corollary \ref{cor5}.

\begin{figure}
\center
\includegraphics[scale = 0.55]{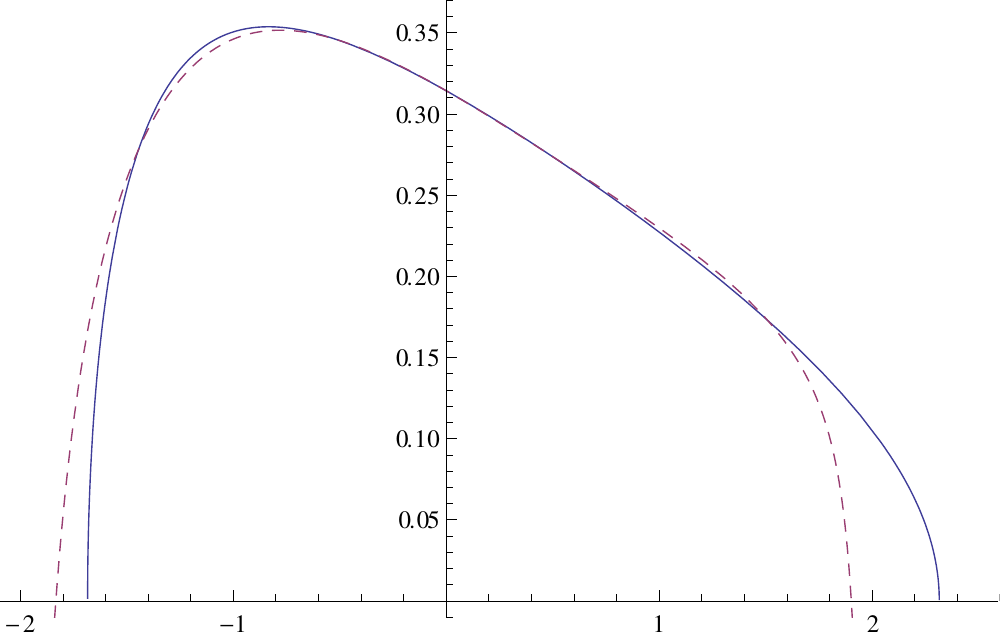}
\q\ \ 
\includegraphics[scale = 0.55]{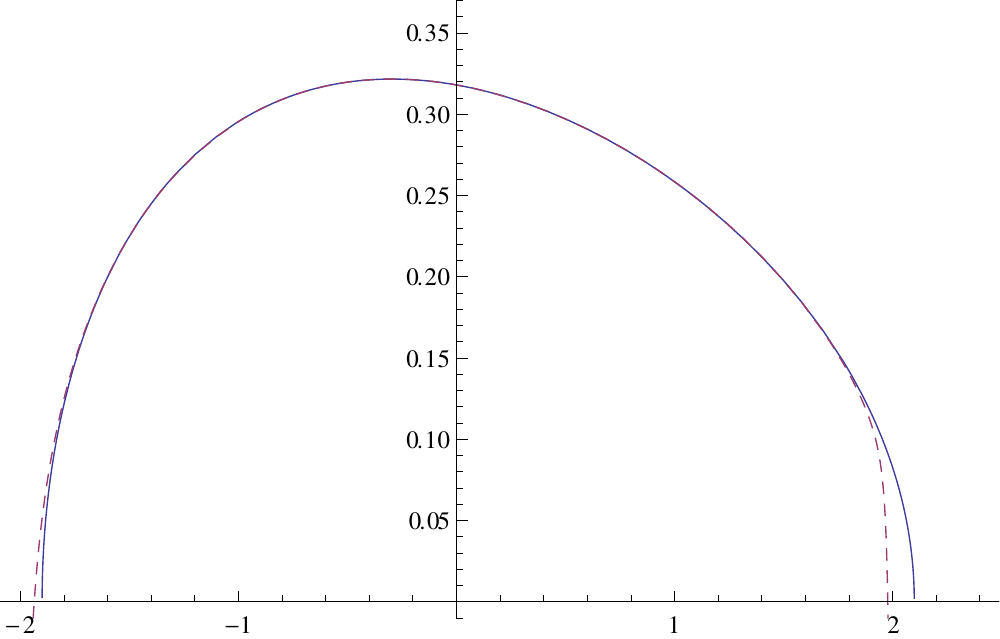}
\caption{Comparison of the asymptotic expansion, shown as the dashed line and the exact result, shown as the solid line, 
$n = 10$ (on the left) and $n=100$ (on the right).}
\label{fig1}
\end{figure} 
}

\end{example}


\section{Preliminaries}

\subsection{Free convolution}
Let us assume that the measure $\mu \in \mes$ has  compact support contained in $ [-L,L]$.
Recall that the Cauchy transform is defined by
$$G_\mu(z) = \int_\mathbb{R} {\frac {\mu(dx)} {z-x}}, \quad z \in \compl^+,$$
which is an analytic function on the  upper half-plane.
A measure is uniquely determined  by its Cauchy transform and
can be recovered from its Cauchy transform by the Stieltjes inversion formula:
\begin{equation}
\label{SIF}
\mu(a,b) = -\frac 1 {\pi} \lim _{y \downarrow 0}\int_a^b \Im G_\mu (x+iy)dx, \q \mu(\{a\}) = \mu(\{b\})=0.
\end{equation}
Since $\mu$ is compactly supported the Cauchy transform has the following power series expansion at $z= \infty$
\begin{equation}
\label{jfjfjf}
G_\mu(z) = \sum_{k=0}^\infty \frac {m_k} {z^{k+1}},
\end{equation}
where $m_k$ are the moments of the measure $\mu$.  Moreover, $|m_k| \leq L^k$. 
It is easy to see that $G_{\mu}(z) = \frac 1 z (1 + o(1))$ at $z = \infty$.
The series (\ref{jfjfjf}) is 
univalent for large $z$ ($|z| >L$) and we can define its functional inverse
$K_\mu(z)$ such that  
$K_\mu(G_\mu(z)) = z,$
which  converges in a neighbourhood of zero. Let us introduce the  function 
\begin{eqnarray}
R_{\mu}(z) = K_\mu (z) - \frac 1 z.
\end{eqnarray}
\label{kfjdshsut}
This function  is called the $R$-transform and can be expressed as formal power series: 
\begin{eqnarray*}
R_\mu(z) = \sum_{l=0}^\infty \kappa_{l+1}z^l,
\end{eqnarray*}
where the coefficients $\kappa_k$ are called the free cumulants of a corresponding measure.
In the case when $m_1 =0$  and $m_2 = 1$ we note that
$\kappa_1 = 0$, $\kappa_2 = 1$, $\kappa_3 = m_3$, $\kappa_4 = m_4 - 2$, $\kappa_5 = m_5 - 5 m_3$.
For  cumulants of higher order the following inequalities have been established in \cite{Kar07}:

\begin{equation}
\label{cum}
|\kappa_l| \leq \frac {2L} {l-1} (4L)^{l-1}, \quad l \geq 2.
\end{equation}

Next, we note some scaling properties of  the Cauchy transform and the $R$-transform. 
We denote by $D_t \mu$ the dilation of a measure $\mu$ by the factor $t$:
$$D_t \mu(A) = \mu (t^{-1}A), \qquad (A \subset \mathbb R \quad \mbox {measurable}).$$
Then the Cauchy transform and  the $R$-transform of the  rescaled measure $D_t \mu$ are
\begin{equation}
\label{dilation}
G_{D_t \mu}(z) =  t^{-1} G_\mu(t^{-1} z )
\q \mbox{and}\q
R_{D_t \mu}(z) = t R_\mu(t z).
\end{equation}

Voiculescu in  \cite{V85} proved that 
for two given compactly supported probability measures $\mu_1$ and $\mu_2$  the $R$-transform of the free convolution 
$\mu_1 \boxplus \mu_2$ is given by the formula
\begin{equation}
\label{r-tr}
R_{\mu_1 \boxplus \mu_2}(z) = R_{\mu_1}(z) + R_{\mu_2}(z),
\end{equation}
on the common domain of these functions. 
 Moreover, (\ref{r-tr}) implies that the free convolution is commutative and associative.

Let us introduce the reciprocal  Cauchy transform
$$F_\mu(z) = 1/G_\mu(z),\q z \in \compl^+ ,$$
which  is an analytic self-mapping of $\mathbb{C^+}$.



Chistyakov and G\"otze 
\cite{ChG11}, Bercovici and Belinschi  \cite{BB07}, Belinschi \cite{Bel08}
proved  the subordination property of free convolution: there exit analytic functions 
$Z_1, Z_2: \compl^+ \to \compl^+$ such that 
$$
\lim_{y \uparrow \infty} \frac {Z_j(i y)}{iy} = 1 , \q j=1,2.
$$
Functions $Z_1$ and $Z_2$ are called subordination functions and satisfy equations:
\begin{eqnarray}
\label{def}
z = Z_1(z) + Z_2(z) - F_{\mu_1}(Z_1(z));
\end{eqnarray}
\begin{equation}
\label{ddeeff}
F_{\mu_1 \boxplus \mu_2}(z) = F_{\mu_1}(Z_1(z)) = F_{\mu_2}(Z_2(z)).
\end{equation}
The next result is due to Belinschi \cite{Bel08} (see Theorem 3.3 and Theorem 4.1).
\begin{theorem}
\label{Belinschi}
Let $\mu_1,\ \mu_2$ be two Borel probability measures on $\real$, neither of them a point mass. 
The following hold:
\begin{enumerate}
\item
The subordination functions from (\ref {def}) and (\ref {ddeeff}) have limits 
$Z_j(x): = \lim_{y\downarrow 0}Z_j(x+i y),$ $ j=1,2,$ $x \in \real.$
\item 
The absolutely continuous part of $\mu_1 \conv \mu_2$ is always nonzero, and its density is analytic wherever
positive and finite, and $F_{\mu_1 \conv \mu_2}$ extends analytically in a neighbourhood of every point 
where the density is positive and finite.
\end{enumerate}
\end{theorem}

\subsection{Semicircle law.} 
The semicircle law plays a key role in free probability. 
The centered semicircle distribution of variance $t$  is denoted by $\omega_t$ and has the density
$$p_{\omega_t}(x) = \frac{1}{2\pi t}\sqrt{(4 t - x^2)_+}, \q x \in \real,
$$
where $a_+: = \max \{a,0\}$.
We denote by $\omega$ the standard semicircle law that has zero mean, unit variance and the density
$$
p_\omega(x) = \frac 1 {2\pi} \sqrt{(4 - x^2)_+},
\q x \in \real.
$$
The  Cauchy transform of $\omega_t$ is given by
$$G_{\omega_t}(z) = \frac {z - \sqrt {z^2 - 4t}}{2t}, \qquad z \in \compl^+.$$
The function $\sqrt {z^2 - 4t}$ is double-valued and has branch points at $z = \pm 2\sqrt t$. We can define two single-valued analytic branches on the complex plane cut along the segment 
$ -2\sqrt t \leq  x \leq 2\sqrt t$ of the real axis. Since the Cauchy transform  has  asymptotic behaviour $1/z$ at infinity, 
we can choose  a branch such that $\sqrt{-1}=i$ on $\compl^+$.
The Cauchy transform $G_{\omega_t}(z)$ has a continuous extension to $\mathbb C^+ \cup \mathbb R$ which acts on $\mathbb R$ by
\begin{eqnarray}
\label{semiext}
\left \{
\begin{array}{rcl}
(x - i \sqrt {4t-x^2})/2t, \qquad \mbox{if}\  |x| \leq 2\sqrt t; \\
(x -  \sqrt {x^2 - 4t})/2t, \ \qquad \mbox{if}\  |x| > 2\sqrt t. \\
\end {array}
\right.
\end{eqnarray}
We see that for every $\delta > 0$, the function $G_{\omega_t}$ can be continued analytically to the domain
$K =  \{x+i y: x \in (-2\sqrt t,2\sqrt t), |y|<\delta\}$ and beyond to the whole Riemann surface 
This analytic continuation is  again denoted by $G_{\omega_t}$. It has the  explicit formula 
$G_{\omega_t}(z) = (z - i \sqrt{4t - z^2})/2t$, where the branch of the square root on $\mathbb C^+$  is chosen such that $ \sqrt{-1} = i$. 
The function $G_{\omega}$ satisfies the functional equation
\begin{equation}
\label{func}
G_{\omega}(z)+ F_{\omega}(z) = z, \quad z\in \compl^+ \cup K.
\end{equation}
One can compute the $R$-transform of the semicircle law: 
$R_{\omega}(z) =  z.$

\subsection{Distance between two measures.}

Below we recall a number of results that we need in the sequel.

We  introduce the Kolmogorov (or uniform) distance between two measures $\mu$ and $\lambda$, which is defined by the formula
\begin{eqnarray*}
d_K(\mu,\lambda) = \sup_{x \in \real}\{|\mu((-\infty,x)) - \lambda((-\infty,x))|\}.
\end{eqnarray*}
%
and the Levy distance between two measures $\mu$ and $\lambda$ is defined by the formula
$$
d_L (\mu,\lambda) =  \inf \{s \geq 0: \mu((-\infty,x - s)) - s \leq \lambda ((-\infty, x)) 
\leq \mu((-\infty, x+s)) + s,\  \forall   x \in \real\}.
$$
The Levy distance is related to the Kolmogorov one by the inequality:
$$
d_L(\mu,\lambda) \leq d_K(\mu,\lambda).
$$

We  need the following result by Voiculescu and Bercovici \cite{BV93} about continuity of free convolutions with respect to the Levy and Kolmogorov distances. 

\begin{theorem}
\label{levy_dist_in}
If $\mu_1,\ \mu_2,\ \nu_1,$ and $\nu_2 \in \mes$, then 
\begin{eqnarray*}
d_L(\mu_1 \conv \nu_1, \mu_2 \conv \nu_2) \leq d_L(\mu_1,\mu_2) + d_L(\nu_1, \nu_2),\\
d_K(\mu_1 \conv \nu_1, \mu_2 \conv \nu_2) \leq d_K(\mu_1,\mu_2) + d_K(\nu_1, \nu_2).
\end{eqnarray*} 
\end{theorem}

The Berry--Esseen type inequality in  free probability was proved by Chistyakov and G\"otze  
\cite{ChG11v2}.   Assume $\mu$ has zero mean, unit variance and finite  third absolute moment $\beta_3$, 
then there exists an absolute  constant $c >0$ such that
\begin{equation}
\label{b-e}
d_K ( \omega, \mu_n) \leq   \frac {c \beta_3}{\sqrt n}, \q n\in \nat.
\end{equation}



\section{A general scheme for asymptotic expansions}

We denote a vector $(\ve_1,\dots,\ve_n) \in \real^n$ by $\unve_n$.
Let us consider a sequence of functions $h_n(\unve_n;t)$, where
$|\ve_j| \leq n^{-1/2}$, $j = 1,\dots,n$ and 
$t \in A \subseteq \real$ (or $\compl$).
Assume that this sequence of functions
satisfies the following conditions:
\begin{equation}
\label{sym}
h_n(\unve_n;t)\ \mbox{is  symmetric  in  all }\ve_j;
\end{equation}
the sequence $h_n$ is compatible, which means 
\begin{eqnarray}
\label{comp}
\lefteqn{
h_{n+1}(\varepsilon_1,\dots,\varepsilon_{j-1},0,\varepsilon_{j+1},\dots\varepsilon_{n+1};t) 
=h_n(\varepsilon_1,\dots,\varepsilon_{j-1},\varepsilon_{j+1},\dots,\varepsilon_{n+1};t),}\nonumber \\
&&\q\q\q\q\q\q\q\q\q\q\q\q\q\q\q\q\q\q\q\q j=1,\dots,n+1;\q\q
\end{eqnarray}
and all  first derivatives vanish at zero:
\begin{equation}
\label{fd}
\frac{\partial}{\partial \varepsilon_j}h_n(\unve_n;t)\Big |_{\varepsilon_j=0} = 0,  \quad j=1,\dots,n.
\end{equation}

Let us denote by  $\set$ $(m \geq n > s \geq 3)$ the set of weight vectors 
$\unve_{m+s}$
where all but $2s$ components are equal to $m^{-1/2}$ and the remaining $2s$ components are bounded by $n^{-1/2}$. 
Let $\alpha=(\alpha_1,\dots,\alpha_m)$ denote an $m$-dimensional multi-index.
Finally, we define 
$$
d^s_r(h,n):=\sup\{|D^\alpha h_{m+s}(\unve_{m+s};t)|:|\alpha|=r,\ 
t \in A,\ \unve_{m+s}\in \set,\ m\geq n \}.
$$
The following proposition from \cite{G85} shows that the limit
$$
h_\infty(\unve_s;t) := \lim_{m \rightarrow \infty} h_{m+s}( m^{-1/2},\dots, m^{-1/2},\unve_s;t), 
\ \  |\ve_j| \leq n^{-1/2},\  j=1,\dots, s
$$
exists.

\begin{proposition}
\label{pr}
Assume $h_{m+s}(\unve_{m+s};t)$, $\unve_{m+s} \in \set$, $m\geq n \geq s\geq 3$, $t \in A$ satisfies conditions $(\ref{sym})-(\ref{fd})$ and the condition
$d^s_3(h,n)<\infty.$
Then   limit
$h_\infty(\unve_s;t)$, $|\ve_j| \leq n^{-1/2}$, $j=1,\dots,s$ exists and the following estimate holds:
$$
|h_{n+s}(n^{-1/2},\dots,n^{-1/2}, \unve_s;t) - h_\infty(\unve_s;t)| \leq cd^s_3(h,n) n^{-1/2},
$$
where $c$ is an absolute constant.
\end{proposition}


We formulate an Edgeworth type expansion  for $h_n(n^{-1/2},\dots,n^{-1/2};t)$
in terms of derivatives of 
$h_{\infty}(\unve_s;t)$ with respect to $\ve_j$, $j=1,\dots,s$ at $\unve_s = 0$.
Below we introduce all necessary  notations.

We  establish ``cumulant"  differential operators $\kappa_p(D)$ via the
formal identity
\begin{equation}
\label{ln}
\sum_{p=2}^{\infty}p!^{-1}\varepsilon^{p}\kappa_p(D)=
\ln\left(1+\sum_{p=2}^{\infty}p!^{-1}\varepsilon ^p D^p\right).
\end{equation}
Expanding in formal power series in the formal variable $\varepsilon$ on the right-hand side of this identity we obtain the definition of the cumulant operators $\kappa_p(D)$.  
Here $D^p$ denotes $p$-fold differentiation with respect to a single variable $\varepsilon$, and $D^{p_1}\cdots D^{p_r} = D^{(p_1,\dots,p_r)}$ denotes differentiation with respect to $r$ different variables 
$\varepsilon_1,\dots,\varepsilon_r$ at the point $\unve_r = 0$.
Since the operators are applied to symmetric functions at zero, 
$\kappa_p(D)$ is unambiguously defined by $(\ref{ln})$.
The first cumulant operators  are $\kappa_2(D) = D^2,\ \kappa_3(D) = D^3,\ \kappa_4(D) = D^4 - 3D^2D^2,$ etc.

Then, we define Edgeworth polynomial operators $P_r(\kappa_.(D))$ by means of the following formal series in $\kappa_r$ and a formal variable $\varepsilon$.
\begin{eqnarray}
\label{134}
\sum_{r=0}^\infty \varepsilon^r P_r({\kappa_.}) = 
\exp \left( \sum_{r=3}^\infty r!^{-1} \varepsilon^{r-2}\kappa_r\right)
\end{eqnarray}
which yields
\begin{eqnarray}
\label{ppp}
P_r({\kappa_.}) = \sum_{m=1}^r m!^{-1} \left\{\sum _{(j_1,\dots j_m)} (j_1 + 2)!^{-1}\kappa_{j_1+2}\cdots 
 (j_m + 2)!^{-1}\kappa_{j_m+2}\right\},
\end{eqnarray}
where the sum $\sum _{(j_1,\dots, j_m)}$ means summation over all $m$-tuples of positive integers  $(j_1,\dots, j_m)$
satisfying $\sum_{q=1}^m j_q=r$ and ${\kappa_.}=(\kappa_3,\dots,\kappa_{r+2})$.
Replacing the variables $\kappa_.$ in $P_r(\cdot)$ by the differential operators 
$$\kappa_.(D): = (\kappa_3(D),\dots,\kappa_{r+2}(D))$$
 we obtain ``Edgeworth" differential operators, say 
$P_r(\kappa_.(D))$.
The following theorem yields an asymptotic expansion for 
$h_n(n^{-1/2}, \dots, n^{-1/2};t)$ (for more details see \cite{G85}).

\begin{theorem}
\label{th}
Assume that $h_{m+s}(\unve_{m+s};t)$, $\unve_{m+s} \in \set$, $m\geq n \geq s \geq 3$, $t \in A$ fulfils conditions $(\ref{sym})-(\ref{fd})$  together with 
\begin{equation}
\label{der1}
d_s^s(h,n) \leq B,
\end{equation}
\begin{equation}
\label{der2}
\sup_{t \in A}\sup_{\unve_{m+s} \in \set}\left|D^{\alpha}h_{m+s}(\unve_{m+s};t)\right| \leq B, 
\end{equation}
where $\alpha  = (\alpha_1,\dots, \alpha_{s-2})$ 
such that 
$$\alpha_i \geq 2,\ \  i=1,\dots, s - 2,\ \  \sum_{i=1} ^{s-2}(\alpha_i - 2) \leq s - 2.$$
Then 
\begin{eqnarray}
\label{expansiondf}
\left| h_n(n^{-1/2}, \dots, n^{-1/2};t) - \sum_{r=0}^{s-3} n^{-r/2}P_r(\kappa_.(D))h_\infty(\unve_r;t)\Big|_{{\unve_r}=0}\right| 
 \leq c_s B n^{-(s-2)/2},
\end{eqnarray}
where $P_0(\kappa_.(D)) = 1$ and  $P_r(\kappa_.(D))$ are given explicitly in $(\ref{ppp})$, 
$c_s$ is an absolute constant.
\end{theorem}

The first four terms of the expansion (\ref{expansiondf})  are
\begin{eqnarray}
\label{expan_2}
\lefteqn{
h_n(n^{-1/2}, \dots, n^{-1/2};t)} \\
&\!\!=&h_\infty(0;t) + \frac 1 {n^{1/2}}\left(\frac{1}{6}
\frac{\partial^3}{\partial \ve_1^3}\right)h_\infty(\ve_1;t)\Big|_{\ve_1=0} \nonumber
\\ 
&\!\!+&\frac{1}{n} \left ( \frac{1}{24}\left(\frac{\partial^4}{\partial\varepsilon_1^4} - 3  \frac{\partial^2}
{\partial\varepsilon_1^2}\frac{\partial^2}{\partial\varepsilon_2^2}\right) + \frac{1}{72} \frac{\partial^3}
{\partial\varepsilon_1^3}\frac{\partial^3}{\partial\varepsilon_2^3}\right)h_\infty(\unve_2;t)\Big|_{\unve_2=0}
 \nonumber
 \\ 
&\!\!+&\frac{1}{48n^{3/2} } \left( \frac {1}{5} \left( \frac{\partial^5}{\partial\varepsilon_1^5} - 10\frac{\partial^3}{\partial\varepsilon_1^3}\frac{\partial^2}{\partial\varepsilon_2^2}\right) \right.  \nonumber
\\
&\!\!+& \frac {1}{3}\left(\frac{\partial^4}{\partial\varepsilon_1^4} - 3 \frac{\partial^2}{\partial\varepsilon_1^2} 
\frac{\partial^2}{\partial\varepsilon_2^2} \right) \!\frac{\partial^3}{\partial\varepsilon_3^3}
+ \left. \frac {1} {27} \frac{\partial^3}{\partial\varepsilon_1^3}\frac{\partial^3}{\partial\varepsilon_2^3} \frac{\partial^3}{\partial\varepsilon_3^3}\right) h_\infty (\unve_3;t) \Big|_{\unve_3=0} \!
+O\left(\frac{1}{n^2}\right)\!. \nonumber
\end{eqnarray}

\begin{remark}
\label{remark}
Conditions (\ref{der1}) and (\ref{der2}) guarantee that the functions 
$$
g^\alpha_{m+s}(\unve_{m+s};t): = D^\alpha h_{m+s}(\unve_{m+s};t),
$$
for $ \alpha  = (\alpha_1,\dots, \alpha_r)$, where  
$ r \leq  s - 3,\  s\geq3$ 
$$\alpha_i \geq 2,\q i=1,\dots, r,\q  \sum_{i=1} ^{r}(\alpha_i - 2) =s-3$$
satisfy the conditions of Proposition \ref{pr}.
In particular, due to Proposition \ref{pr}  for every $\alpha$ the functions
$g^\alpha_{m+s}(\unve_{m+s};t)$ converge to $g^{\alpha}_{s}(\unve_{s};t)$ as $m\to \infty$ uniformly in $\unve_{s}$, $|\ve_j| \leq n^{-1/2}$, $n \geq 1$, $j=1,\dots,s$  and due to Theorem \ref{uniform} (see Appendix) we conclude that $g^{\alpha}_{s}(\unve_{s};t) = D^{\alpha}h_\infty(\unve_s;t)$, $r \leq s-3$.
\end{remark}



\section{Proofs of results}

\subsection{Truncation}
We assume that $\mu \in \mes$ satisfies $\beta_9 < \infty$.
Consider free random variables $\hat X_1,  \hat X_2, \dots$ with distribution $\hat \mu \in \mes$ such that 
$\hat\mu([-n^{1/3}, n^{1/3}]) = 1$ and $\hat \mu (B) = \mu(B)$ for all Borel sets 
$B \subseteq [- n^{1/3},  n^{1/3}] \backslash \{0\}$.
Denote by $\hat \mu_n$ the distribution of the random variable 
$\hat Y_n : = (\hat X_1 + \dots + \hat X_n)/\sqrt n$.
We denote by $\hat m_k$ and $\hat \beta_k$, $k=1,2,\dots$ moments and absolute moments of  $\hat \mu$.
Let us also introduce the centered random variables 
\begin{eqnarray*}
X^*_1: = \frac {\hat X_1 - a_n}{b_n}, \q X^*_2 := \frac {\hat X_2 - a_n}{b_n},\q\dots\q \mbox{and} \q
Y^*_n: = \frac 1 {\sqrt n} \sum_{j=1}^n X^*_j,
\end{eqnarray*}
where
\begin{eqnarray*}
a_n := \hat m_1,\q 
b^2_n := \int (x - a_n)^2\hat \mu(dx) =  1 - \int_{A_n} x^2\mu(dx) - a_n^{2}, \ \ A_n: = \{|x| > n^{1/3}\}.
\end{eqnarray*}
Denote by $\mu^*$ and $\mu_n^*$ the distributions of the random variables $X_1^*$ and $Y_n^*$ respectively,
and by $m_k^*$ and  $\beta_k^*$, $k=1,2,\dots$ moments and  absolute moments of  $\mu^*$.
Note that $m_1^* = 0$ and $m_2^* = 1$.
Due to the assumption  $\beta_9 < \infty$ 
we have
\begin{eqnarray*}
|a_n| \leq n^{-8/3}\int_{A_n}|x|^9\mu(dx) \leq  \beta_9 n^{-8/3}
\end{eqnarray*}
and
\begin{eqnarray*}
|1 - b_n^2| \leq \beta_9^2 n^{-16/3} + \beta_9 n^{-8/3}.
\end{eqnarray*} 
By above inequalities we obtain
\begin{eqnarray}
\label{supp}
b_n^{-1}(n^{1/3} + |a_n|) \leq  n^{1/3}.
\end{eqnarray}
By (\ref{supp}) we conclude that  the support of $\mu^*$ is contained in $[- n^{1/3},  n^{1/3}]$.

Furthermore, we deduce that 
\begin{eqnarray*}
\beta^*_3 \leq b_n^{-3}\left(\hat \beta_3 + 3|a_n|\hat \beta_2+3|a_n|^2\hat \beta_1 + |a_n|^3\right)
\leq b_n^{-3}\hat \beta_3 +  4 \beta_9 n^{-8/3}.
\end{eqnarray*}


Let $\hat \omega$ be a semi-circle distribution with mean $\sqrt n a_n$ and variance $b^2_n$.
By the triangle inequality we get
\begin{eqnarray}
\label{triangl}
d_K(\mu_n,\omega) \leq d_K (\mu_n , \hat \mu_n) + d_K ( \hat \mu_n, \hat \omega) + 
d_K(\hat \omega, \omega).
\end{eqnarray}
By Theorem \ref{levy_dist_in}
the first term in the right hand side has a bound
\begin{eqnarray}
\label{ineq1}
 d_K (\mu_n , \hat \mu_n) \leq n \mu (\{|x| > \sqrt [3] n\} \leq \frac  {\beta_9} {n^{2}}.
\end{eqnarray}
We find the an estimate for the last term in (\ref{triangl})
\begin{eqnarray}
\label{ineq2}
d_K (\hat \omega, \omega) \leq c \left( \frac 1 {b_n} + \frac {\sqrt n a_n} {b_n} - 1\right)
\leq \frac c { n^2}.
\end{eqnarray}
Note, that we have an equality
\be*
d_K (\hat \mu_n, \hat\omega) = d_K (\mu_n^*, \omega).
\ee*
Our next aim is to apply the general asymptotic scheme for expansion to $\mu^*_n$.

\subsection {Application of the general scheme for asymptotic expansions}

Below we apply the general scheme  to compute the asymptotic expansions for $p_{\mu^*_n}$ and $\mu_n$.
We set $h_n(n^{-1/2}, \dots, n^{-1/2};z): = G_{\mu^*_n}(z)$, $z \in K$,  where $G_{\mu^*_n}(z)$ is the extension defined in Corollary \ref{extension_g} (in our case $\mu^* = \mu$).
Let us introduce two notations:
\be*
\muves & := & D_{\ve_1}\mu^* \conv \dots \conv D_{\ve_{m+s}}\mu^*,\q \unve_{m+s}\in \set,\ \ m \geq n,\\
\mu_s^{(\unve_s)} & : = & D_{\ve_1}\mu^* \conv \dots \conv D_{\ve_s}\mu^*,
\q|\ve_j| \leq n^{-1/2},\q j= 1, \dots, s.
\ee*

The following results follow from Theorem \ref {extension_33} and allow for an easy application of the expansion scheme.

\begin{corollary}
\label{lklklk}
For every $\delta \in (0,1/10)$ and $n \geq c (\mu^*,s) \est$ the Cauchy transform 
$G_{\mutilda}$ has an analytic continuation to $K$
such that
\begin{equation}
\label{cont41}
G_{\mutilda}(z) = G_{\omega} (z) +  l _{\unve_s}(z), \ \  z \in K,
\end{equation}
where 
$| l _{\unve_s}(z)| \leq \frac {c(s)}{n \sqrt \delta}$ on $K$.
\end{corollary}

\begin{remark}
In (\ref{cont41})  we understand $G_{\omega} (z)$  as an analytic continuation of the corresponding Cauchy transform, which is defined in the following way:
$$G_{\omega}(z) = (z - i \sqrt{4 - z^2})/2,\ \ z \in K.$$
\end{remark}

\begin{corollary}
\label{dfdfdfdfd}
For every $\delta \in (0,1/10)$ and 
$m \geq n \geq c(\mu^*,s) \est$
the Cauchy transform  $G_{\muves}$ has an analytic continuation to $K$ such that
\begin{eqnarray}
G_{\muves}(z) = G_{\mutilda} (z) +  l(z),
\q z \in K,
\end{eqnarray} 
where
$| l(z)| \leq \frac {c(s)} {\sqrt {\delta n}}$
on $ K$.
\end{corollary}

\begin{corollary}
\label{sym_comp}
For every $\delta \in (0, 1/10)$,   $m \geq n \geq c(\mu^*,s)\est$
the analytic continuation $G_{\muves}(z)$, $z \in K$ is a
symmetric and compatible function of $\ve_j$, $j=1,\dots, m + s$.
\end{corollary}

\begin{theorem}
\label {asdfasdf}
For every $\delta \in (0, 1/10)$,   $m \geq n \geq c(\mu^*,s)\est$
the analytic continuation of $G_{\muves}(z)$, $z \in K$ is a smoothly differentiable function of variables $\ve_j$, $j=1, \dots, 2s$.
(Here we mean those $2s$ variables which are not fixed and just bounded by $n^{-1/2}$).
Moreover, the inequality holds:
\begin{eqnarray*}
\sup_{z \in K} \sup_{\unve_{m+s} \in \set} \left |D^\alpha G_{\muvesthree}(z)\right |\leq c, \q |\alpha| \leq s.
\end{eqnarray*}
\end{theorem}

\begin{theorem}
\label{vanish_dir}
For every $\delta \in (0, 1/10)$,   $m \geq n \geq c(\mu^*,s)\est$ and  $z \in K$
\begin{eqnarray*}
\frac {\partial} {\partial \ve_j} G_{\muves} (z) \Big|_{\ve_j = 0} =0,\q j=1,\dots,2s.
\end{eqnarray*}
\end{theorem}


In view of the above results, we can choose the sequence of  extensions of the Cauchy transforms 
$G_{\muves}(z)$, $z \in K$ as the sequence of functionals
$h_{m+s}(\unve_{m+s};z)$, i.e. 
\begin{eqnarray*}
h_{m+s}(\unve_{m+s};z) :  = G_{\muves}(z), 
\ \ z \in K,\ \ m \geq n \geq c(\mu^*,s) \est,
\end{eqnarray*}
and
$$
h_\infty(\unve_s;z): = G_{\mutilda}(z), \ \ z\in K,\ \ n \geq c(\mu^*,s)\est.
$$

Now we can apply the general scheme and compute the expansion for $G_{\mu^*_n}$ in terms of derivatives of $G_{\mutilda}$ with respect to $\ve_j$, $j=1, \dots, s$ at $\unve_s = 0$.



\subsection{Positivity of the density of $\muves$.}

Our aim is to find an interval where the density of $\muves$ is positive.
The main idea is based on the Newton-Kantorovich Theorem (see Theorem \ref{NK}, for a proof see \cite{Kant96}).

%

Let us consider a pair of measures $\nu_1$ and $\nu_2$. We can rewrite the equations (\ref{def}) and (\ref{ddeeff})
as a system
\begin{eqnarray}
\left \{
\begin{array}{rcl}
\label{system_eq2}
(z - Z_1(z) - Z_2(z))^{-1} + G_{\nu_1}(Z_1(z))  = 0\  \\
(z - Z_1(z) - Z_2(z))^{-1} + G_{\nu_2}(Z_2(z))  = 0,\\
\end {array}
\right.
\end{eqnarray}
where $G_{\nu_1}$ and $G_{\nu_2}$ are the  Cauchy transforms of $\nu_1$ and $\nu_2$, correspondingly.
Choose another pair of  measures $\mu_1$ and $\mu_2$ such that the Levy distance between $\nu_j$ and $\mu_j$ is sufficiently small for $j =1,2$. 
Then we can define subordination functions for the couple $(\mu_1, \mu_2)$ as a solution  of (\ref{system_eq2}),
where $G_{\nu_1}$ and $G_{\nu_2}$ are replaced by the Cauchy transforms of $\mu_1$ and $\mu_2$ correspondingly. Denote these subordination functions by 
$t^0_1$ and $t^0_2$.
According to the Newton-Kantorovich Theorem  one can show that the
subordination functions $Z_j$ and $t^0_j$, $j=1,2$ are sufficiently close to each other.
We can choose $\mu_1$ and $\mu_2$ to be equal, so that $t^0_1 =  t^0_2$. 
Such a choice essentially simplifies the structure of equations (\ref{def}) and (\ref{ddeeff}).

%

Let us prove one result about the Levy distance.
\begin{lemma}
\label{ocenka_levy}
Assume that $\mu,\ \nu \in \mes$ and $\mu$ has zero mean and unit variance.
Then
$d_L(\nu , \nu\conv \mu_s^{(\unve_s)}) \leq  2 s n^{-1/3}$, $|\ve_j| \leq 1/\sqrt n$, $j=1,\dots,k$.
\end{lemma}

\begin{proof}
 From Theorem \ref{levy_dist_in}, we get 
\begin{eqnarray*}
\label{levy}
d_L(\nu, \nu \conv \mu_s^{(\unve_s)})  \leq d_L(\delta_0, \mu_s^{(\unve_s)})
 \leq \sum_{i = 1}^s d_L(\delta_0, D_{\ve_i}\mu),
\end{eqnarray*}
where $\delta_0$ is a delta function. By the Chebyshev inequality we obtain
$$
 d_L(\delta_0, D_{\ve_i}\mu) \leq 2  \ve_j^{2/3} \leq 2 n^{-1/3}.
$$
Hence,
$
d_L(\nu , \nu\conv \mu_s^{(\unve_s)}) \leq    2 s n^{-1/3}.
$
\end{proof}


In the sequel we need the following estimates for $G_{\omega}$.
\begin{lemma}
\label{sets}
For  every $\delta \in(0,1/10)$ we define the set
$$
K_\delta= \{x+iy: x \in [-2 + \delta,\  2 - \delta], |y| \leq 2\delta \sqrt \delta\}.
$$
Then, we have $G_{\omega}(K_{\delta}) \subset
D_{\theta,1.4} = \{z \in \compl^-: \mbox{arg}\ z \in (-\pi +\theta, -\theta);|z|<1.4\},
$
where the angle $\theta = \theta(\delta)$ is chosen in such a way  that
$ 2\sin \theta = \sqrt {\frac {\delta} {4} \left( 1 - \frac {\delta} {4}\right)}.$
\end{lemma}

\begin{proof}
Figure $\ref{fig3}$ illustrates the sets $K_\delta$ and $D_{\theta,1.4}$.

\begin{figure}[h]
\center
\begin{tikzpicture} [scale = 0.8]
    \draw [-latex](-3,0) -- (3,0);
    \draw [-latex](0,-1.5) -- (0,1);
    \draw (-2,0.2) -- (2,0.2);
    \draw (-2,0.2) -- (-2, -0.2);
    \draw (-2,-0.2) -- (2,-0.2);
    \draw (2,0.2) -- (2, -0.2); 
    \draw (0.5, 0.7) node {$K_\delta$};
    \draw (2.2, 0.3) node {$\delta$};
    \draw (-2.2, 0.3) node {$\delta$};
    \draw (-2.4 cm, 2pt ) -- (-2.4 cm, -2pt) node[anchor = north] {$-2$};
    \draw (2.4 cm, 2pt ) -- (2.4 cm, -2pt) node[anchor = north] {$2$};

    \draw [-latex](5,0) -- (11,0);
    \draw [-latex](8,-2.5) -- (8,1);
    \draw (6, -1) arc (210:330:2.3);
    \draw (8,0) -- (6,-1);
    \draw (8,0) -- (10,-1);
    \draw (7.5,0) arc (180:210:0.46);
    \draw (8.5,0) arc (360:330:0.46);
    \draw (8.8,-0.2) node {$\theta$}; 
    \draw (7.2,-0.2) node {$\theta$}; 
    \draw (8.5, -2.4) node {$-1.4$};
    \draw (8.7,- 1) node {$D_{\theta,1.4}$};
    \draw (7.9, -2.17 cm ) -- (8.1, -2.17cm);
\end{tikzpicture}
\caption{}
\label{fig3}
\end{figure}

First we show that $G_\omega(K_\delta) \subseteq D_{\theta,1.4}$, 
where  $G_\omega$ is an analytic extension of the Cauchy transform of $\omega$ on $K_\delta$.
Fix a point $z_0 \in K_\delta$, and write $G_\omega(z_0) = R e^{i \psi}$. 
In order to prove  $G_\omega(z_0) \in D_{\theta, 1.4}$ we need to verify that 
$|\sin \psi| > \sin \theta$ and $R < 1.4$.
From the functional equation (\ref{func}) we have
$$
\left(R+\frac 1 R\right)\cos\psi + i \left(R - \frac 1 R\right) \sin \psi = z_0.
$$
From $|\Re z_0| \leq 2 -  \delta$, we get
$
2|\cos \psi| \leq \left(R+ \frac 1 R\right)|\cos \psi| \leq  2 - \delta.
$
This implies 
$|\cos \psi| \leq 1 - \delta/2,$
hence
\begin{eqnarray*}
|\sin \psi| & =  & 
\sqrt{1 - \cos ^2 \psi} \geq \sqrt {1 - \left(1 - \delta/2\right)^2}
= 
 \sqrt{\delta/4\left(1 - \delta/4\right)}
>
\sin \theta.
\end{eqnarray*}
Thus we obtain the desired result $|\sin \psi| > \sin \theta$.

In order to estimate $R$ we consider the imaginary part of $z_0$
\begin{eqnarray*}
\label{estim1}
2 \delta \sqrt \delta> |\Im z_0| = |\sin \psi| \left|R -\frac 1 R\right| > 
  \frac {|R^2 -1|} R  {\frac {\sqrt \delta} 2}.
\end{eqnarray*}
If $R>1$,  we get the inequality $ R^2 - 4 \delta R - 1 < 0$.
Therefore, $R$ must be bounded from above by the intercept of the positive $x$-axis and 
the parabola $ y = R^2 - 4  \delta R -1.$
The roots of the equation $R^2 - 4  \delta R - 1 = 0$ are
$$
R = 2  \delta \pm \sqrt{4 \delta^2 + 1}.
$$
By  the choice of $\delta$ we have 
$2 \delta + \sqrt{4 \delta^2 + 1} < 1.22$.
This implies  $R < 1.4.$ 
\end{proof}

The following inequalities are due to Kargin \cite{Kar13}.

\begin{lemma}
\label{K_ineq}
Let $d_L(\mu_1, \mu_2) \leq p$ and $z = x + i y$, where $y > 0$.
Then
\begin{enumerate}
\item
$|G_{\mu_1} (z) - G_{\mu_2} (z)| < \widetilde c p y^{-1} \max \{1,\  y^{-1}\}$,
where $\widetilde c >0$ is a numerical constant;
\item
$|\frac {d^r}{dz^r}(G_{\mu_2} (z) - G_{\mu_1} (z))| < \widetilde c_r p y^{-1 - r} \max \{1,\  y^{-1}\}$,
where $\widetilde c_r >0$ are  numerical constants.
\end{enumerate}
\end{lemma}

Consider a pair of measures $(\nu_1, \nu_2)$ and  introduce  a function
$F(t): \compl^2 \to \compl^2$ by the formula
\begin{eqnarray*}
\label{lskdjf}
F(t) = 
\left(
\begin{array}{c}
(z - t_1 - t_2)^{-1} + G_{\nu_1}(t_1)\\
(z - t_1 - t_2)^{-1} + G_{\nu_2}(t_2)
 \end{array}
\right).
\end{eqnarray*}
The equation
$F(t) = 0$
has a unique solution,  say $Z = (Z_1(z),Z_2(z))$, where $Z_1(z)$ and 
$Z_2 (z)$ are subordination functions.
Let  $(\mu_1,\mu_2)$ be  another pair of measures. Assume  $t^0 = (t^0_1,t^0_2) = (t^0_1(z), t^0_2(z))$ solves the  system of equations
\begin{eqnarray*}
\left \{
\begin{array}{rcl}
(z - t^0_1 - t^0_2)^{-1} + G_{\mu_1}(t^0_1)= 0\  \\
(z - t^0_1 - t^0_2)^{-1} + G_{\mu_2}(t^0_2)= 0.\\
\end {array}
\right.
\end{eqnarray*}
Then $F(t^0)$ has the form
$$F(t^0) = 
\left(
\begin{array}{cc}
 G_{\nu_1}(t^0_1) - G_{\mu_1}(t^0_1)\\
 G_{\nu_2}(t^0_2) - G_{\mu_2}(t^0_2)
 \end{array}
\right).
$$
The derivative of $F$ with respect to $t$  at $t^0$ is

\begin{eqnarray*}
F'(t^0) 
=
\left(
\begin{array}{cc}
G'_{\nu_1}(t^0_1) + G^2_{\mu_1}(t^0_1)  &  G^2_{\mu_1}(t^0_1) \\
G^2_{\mu_2}(t^0_2)&   G'_{\nu_2}(t^0_2) + G^2_{\mu_2}(t^0_2)
\end{array}
\right).
\end{eqnarray*}
The inverse matrix of $F'(t^0)$ is
\begin{eqnarray}
\label{dfgghjdgh}
[F'(t^0)]^{-1} = \frac 1 {\det[F'(t^0)]}
\left(
\begin{array}{cc}
G'_{\nu_2}(t^0_2) + G^2_{\mu_2}(t^0_2) &  - G^2_{\mu_1}(t^0_1) \\
- G^2_{\mu_2}(t^0_2)&   G'_{\nu_1}(t^0_1) + G^2_{\mu_1}(t^0_1)
\end{array}
\right),
\end{eqnarray}
where 
\be*
\label{det}
\det[F'(t^0)] = 
(G'_{\nu_2}(t^0_2) + G^2_{\mu_2}(t^0_2) )(G'_{\nu_1}(t^0_1) + G^2_{\mu_1}(t^0_1)) - G^2_{\mu_1}(t^0_1)G^2_{\mu_2}(t^0_2).
\ee*
After simple computations, we obtain
\begin{eqnarray}
\label{59}
[F'(t^0)]^{-1}F(t^0) 
  = \frac 1 {\det[F'(t^0)]}
\left(
\begin{array}{c}
(G'_{\nu_2}(t^0_2) + G^2_{\mu_2}(t^0_2))S_1 (t^0_1) - G^2_{\mu_1}(t^0_1)S_2 (t^0_2)\\
(G'_{\nu_1}(t^0_1) + G^2_{\mu_1}(t^0_1))S_2 (t^0_2)
- G^2_{\mu_2}(t^0_2)S_1 (t^0_1) 
\end{array}
\right),
\end{eqnarray}
where $S_j (t^0_j): =  G_{\nu_j}(t^0_j) - G_{\mu_j}(t^0_j)$.

The second derivative of $F$ with respect to $t$ at $t_0$ is
\begin{eqnarray}
\label{xvcbvbnbnm}
F''(t^0)  =  
\left(
\begin{array}{cccc}
D_1(t^0_1) &2 G^3_{\mu_2}(t^0_2)& 2 G^3_{\mu_1}(t^0_1)&2 G^3_{\mu_2}(t^0_2)\\
 2 G^3_{\mu_1}(t^0_1)&2 G^3_{\mu_2}(t^0_2)&2 G^3_{\mu_1}(t^0_1)&D_2(t^0_2)
\end{array}
\right),
\end{eqnarray}
where $D_j(t^0_j): = G''_{\nu_j}(t^0_j) - 2 G^3_{\mu_j}(t^0_j)$.

\begin{proposition}
\label{bla-bla}
Let  $\nu_1, \nu_2 \in \mes$ be measures
neither of them being a point mass.
Then for every  $\delta \in (0, 1/10)$ there exists $c$ such that if
$d_L(\omega,  \nu_1\conv \nu_2)\leq c \delta^2$ then 
the density $p_{\nu_1\conv \nu_2}(x)$ is positive and analytic  on  $[- 2 + \delta, 2 - \delta]$.
\end{proposition}

\begin{proof}
We  would like to find an interval where the density is positive. 
To this end, define a subordination function $Z_{\omega_{1/2}}(z)$ which solves the equations
$$z = 2Z_{\omega_{1/2}}(z) -   F_{\omega_{1/2}}(Z_{\omega_{1/2}}(z)) \q \mbox {and} \q
F_{\omega}(z) = F_{\omega_{1/2}}(Z_{\omega_{1/2}}(z)).$$
Solving this equations obtaining we obtain
\begin{eqnarray*}
Z_{\omega_{1/2}}(z) = \frac {3z +  \sqrt{z^2 - 4}}{4},
\end{eqnarray*}
and an analytic continuation of $Z_{\omega_{1/2}}$ to $K_\delta$ is given by
\begin{eqnarray*}
Z_{\omega_{1/2}}(z) = \frac {3z + i \sqrt{4 - z^2}}{4}.
\end{eqnarray*}
It easy to see that the following inequality holds:
$$\Im Z_{\omega_{1/2}}(x) > \sqrt {\delta}/3,\q x \in [-2 +\delta, 2 - \delta].$$
On  $\compl^2$ 
we choose the norm:
$$\|(z_1,z_2)\| = \sqrt{|z_1|^2 + |z_2|^2}.$$
Now we apply the Newton-Kantorovich Theorem (see Theorem \ref{NK}) to the equation
$F(t) = 0$ for 
$z \in M:= \{x + i y : x \in [-2 + \delta, 2 - \delta],\   0 < y < \delta \sqrt \delta \}$.
In formulas (\ref{dfgghjdgh}), (\ref{59}) and (\ref{xvcbvbnbnm}) we set $\mu_1 = \mu_2 =  \omega_{1/2}$ 
and $t_1^0 = t_2^0 = Z_{\omega_{1/2}}$. 
Since $|Z_{\omega_{1/2}}(z)| < 2$, $z \in M$, we choose the branch of $G_{\omega_{1/2}}$ such that
$
G_{\omega_{1/2}}(z) = z - i \sqrt{2 - z^2}, 
$
$|z| <2$.

1. First, we estimate $\|[F'(t^0)]^{-1}\|$.
We  computed $\det[F'(t^0)]$ above.
Moreover, due to Lemma \ref{K_ineq} with $p: = c \delta^2$ we have
$G'_{\nu_j}(t^0_j) = G'_{\omega_{1/2}}(t^0_j) + f_j(t^0_j)$, where 
$|f_j(t_j^0)| \leq \tilde c_1 p \delta^{-3/2}$ on $M$, $j=1,2$.
Hence, 
\begin{eqnarray*}
\lefteqn{
\det [F'(t^0)] } \\ &\!\!=&\!\!\!
(G^2_{\omega_{1/2}}(t^0_2) + G'_{\omega_{1/2}}(t^0_2)  + f_2(t_2^0) )( G^2_{\omega_{1/2}}(t^0_1)+ G'_{\omega_{1/2}}(t^0_1) + f_1(t_1^0)) 
 -  G^2_{\omega_{1/2}}(t^0_1)G^2_{\omega_{1/2}}(t^0_2) \\
& = & g (z) +(f_1(t_1^0)+f_2(t_1^0)) (G'_{\omega_{1/2}}(t^0_1) +  G^2_{\omega_{1/2}}(t_1^0))+f_1(t_1^0)f_2(t_1^0),
\end{eqnarray*}
where 
$$
g(z) = \left( G^2_\omega (z) + G'_{\omega_{1/2}}(Z_{\omega_{1/2}}(z))\right)^2 -\  G^4_{\omega}(z).
$$
We find that
\be*
\label{derivative}
G'_{\omega_{1/2}}(Z_{\omega_{1/2}}(z)) = 1 + \frac {iZ_{\omega_{1/2}}(z)}{\sqrt{2 - Z_{\omega_{1/2}}(z)}} =
1+\frac{3 i z-\sqrt{4-z^2}}{\sqrt{p(z)}},
\ee*
where $p(z): = 36-10 z^2-6 i z \sqrt{4-z^2}$.
The function $p(z)$ has  zeros at $\pm 3/ \sqrt 2$,
hence $|G'_{\omega_{1/2}}(Z_{\omega_{1/2}}(z))|$ is uniformly bounded on $M$.
Finally, we obtain
$$
g(z) = 
\left(1+\frac{3 i z-\sqrt{4-z^2}}{\sqrt{p(z)}}\right)
\left(1+\frac{3 i z-\sqrt{4-z^2}}{\sqrt{p(z)}} +\frac{1}{2} \left(z-   i \sqrt{4-z^2}\right)^2 \right). 
$$

First  of all we estimate $|g(z)|$ on an interval $ [-2 + \delta, 2 - \delta]$.
Obviously,
$$
 \frac{3 i x-\sqrt{4 - x^2}}{\sqrt{36 - 10 x^2 - 6 i x \sqrt{4 - x^2}}}
=\frac { 2 i x \sqrt{4 - x^2} - 3}{9 - 2x^2}
$$

and hence 
\be*
g(x)  = h_1(x)
\left(h_1(x) + x^2 - 2 - 2ix \sqrt{4 - x^2}\right),
\ee*
where 
$$h_1(x): = \frac {6 - 2 x^2 + 2 i x \sqrt{4 - x^2}}{9 - 2 x^2},\q
\left|h_1(x)\right| = 
\frac 2 {\sqrt{9 - 2x^2}} \geq \frac 2 3
$$
and
\be*
\left|h_1(x) + x^2 - 2 - 2ix \sqrt{4 - x^2}\right|
= \frac {2 \sqrt{4 - x^2}}{\sqrt {9 - 2x^2}} \geq c_1 \sqrt \delta
\ee*
for $x\in  [-2 + \delta, 2 - \delta]$. 
We conclude that
$
|g(x)| \geq c_2 \sqrt \delta,\ \  x \in  [-2 + \delta, 2 - \delta].
$

In order to estimate $|g(z)|$ on $M$ 
we expand $g(x+iy)$ with respect to $y$ at zero:
\be*
g(x+ iy) = g(x) + R(x,y), \q x \in  [-2 + \delta, 2 - \delta],\q 0 < y < \delta \sqrt \delta,
\ee*
where $R(x,y)$ is a remainder term such that 
$$|R(x,y)| \leq 
\max_{\genfrac {}{}{0pt}{}{x \in [-2 + \delta, 2 - \delta]}{0< y < \delta \sqrt \delta}} 
|g'(x+iy)| \delta \sqrt \delta.$$

We find that $g'(z) = g_1(z)/g_2(z)$, where
\be*
g_1(z) = &-& 1488
+4 \left(2 z^6-28 z^4+186 z^2 -  \left(z^4 - 9 z^2  - 9\right) \sqrt{(4 -  z^2)
p(z)}\right.\\
& - & \left. 2 i z \left(  z^4 - 12  z^2+ 7 \right) \sqrt{4 - z^2} - i z \left( z^4 - 11 z^2 + 39 \right) \sqrt{p(z) }\right),\q\q
\ee*
$$
g_2(z) =  i\left(p(z)/2\right)^2\sqrt{4 - z^2}.$$
We conclude that $|g'(z)| \leq c_1 / \sqrt \delta$, $z \in  M$. 
Hence $|R(x,y)| \leq c_1 \delta$ and
 $|g(z)| \geq c_2 \sqrt \delta$, $z \in M$.
Then 
$|\det [F'(t^0)]| \geq ||g(z)| -  c_3 p \delta^{-3/2}| \geq \sqrt \delta (c_2 - c_3 p \delta^{-2})$.
We can choose $c$ in $p = c \delta^2$ such that
$\|[F'(t^0)]^{-1}\| \leq c_4 \delta^{-1/2}=: \beta_0$, $z \in M$.

2. We  estimate $\|[F'(t^0)]^{-1}F(t^0)\|$. Due to Lemma \ref{K_ineq} we arrive at
\begin{eqnarray*}
\|[F'(t^0)]^{-1}F(t^0)\|  < 
c_1  p \delta^{-3/2} = : \eta_0,\q z \in M.
\end{eqnarray*}

3. At last, we estimate $\|F''(t^*)\|$, where $t^* = (t_1^*, t_2^*)$ such that $\|t^* - t^0\| \leq 2 \eta_0$.
Note $2\eta_0 < \sqrt \delta /3$, guarantees  that
$\Im t^*_j(z)>0$, $z \in M$, $j=1,2$. 
Furthermore, note that  $|G_{\omega_{1/2}}(z) |\leq \sqrt 2$ for  $z \in \compl^+ \cup \real$.
Due to Lemma \ref{K_ineq} the following estimate holds:
$$
\|F''(t^0)\| \leq \max\{|G''_{\nu_j}(Z_{\omega_{1/2}}) - 2G^3_{\omega_{1/2}}(Z_{\omega_{1/2}})|,\ 2 |G^3_{\omega_{1/2}}(Z_{\omega_{1/2}})|,\ j=1,2\}.
$$
 Lemma \ref{K_ineq} implies
$$
G''_{\nu_j}(Z_{\omega_{1/2}}) = G''_{\omega_{1/2}}(Z_{\omega_{1/2}}) + f(Z_{\omega_{1/2}}),\q
j=1,2,
$$
where $|f(Z_{\omega_{1/2}})| \leq \tilde c_2 p \delta^{-2}$ on $M$.
Let us estimate $G''_{\omega_{1/2}}(Z_{\omega_{1/2}})$ on $M$. We find that
\be*
G''_{\omega_{1/2}}(Z_{\omega_{1/2}}(z)) 
&\!\! = \!\!&
2i(2 - Z^2_{\omega_{1/2}}(z))^{-3/2} = 
2i\left(\!2+ \frac {( \sqrt{4 - z^2} -3iz)^2} {16} \!\right)^{\!-3/2}\\
&\!\! = \!\! &\frac i {4  (p(z))^{3/2}}=
 \frac {i(3\sqrt{4-z^2} + iz)^3} {256  (9 - 2z^2) ^{3}}.
\ee*
Then  the bound 
$|G''_{\omega_{1/2}}(Z_{\omega_{1/2}}(z))| \leq c_2 $ holds on $M$. 
Choosing $p = c \delta^2$
we conclude that
$
\|F''(t^0)\|  \leq c_3 =: K_0.
$

The function $(z - t_1^*(z) - t_2^*(z))^{-3}$ is continuous for $z \in M$ because $\Im t^*_j(z)>0$, $j=1,2$. It follows that the  estimate for the second derivative holds for $t^*$ such that $\|t^* - t^0\|< 2\eta_0$, $z \in M$.


The Newton-Kantorovich Theorem (see Theorem  \ref{NK}) yields us that if $\beta_0$, $\eta_0$ and $K_0$  satisfy the inequality
$h_0 := \beta_0 \eta_0 K_0 \leq 1/2$,  then the equation $F(t) = 0$ has the unique solution $(Z_1(z), Z_2(z))$ 
in a ball 
$$B_0: = \left\{t \in \compl^2: \|t - t^0\| \leq \frac {1 - \sqrt{1 - 2 h_0}}{h_0}\eta_0\right \}.$$
It means that
$$|Z_{\omega_{1/2}}(z) - Z_j(z)| \leq  \frac {1 - \sqrt{1 - 2h_0}}{h_0} \eta_0 
= c_4 p  \delta^{-3/2}, \ \ j=1,2,\ \ z \in M.
$$
Finally, we derive the following bound for the Cauchy transform
\begin{eqnarray*}
\left|\frac 1 {z - 2 Z_{\omega_{1/2}}(z)} - \frac 1 {z -  Z_1(z) -  Z_2(z)} \right|
< \frac {2|G^2_\omega(z)| c_5 p \delta^{-3/2}}{|1 - 2 c_5 p \delta^{-3/2}|G^2_\omega(z)||},
\end{eqnarray*}
\begin{equation}
\label{inequality}
|G_\omega(z) - G_{\nu_1 \conv \nu_2}(z)|< c_6 p \delta^{-3/2},\q z \in M.
\end{equation}
Due to Theorem \ref {Belinschi} the limits 
$Z_j (x): = \lim_{y\downarrow 0}Z_j(x+i y)$,  $x \in [-2 + \delta, 2 - \delta]$, $j=1,2$  exist.
Hence
the limit 
$G_{\nu_1 \conv \nu_2}(x) : =\lim_{y\downarrow 0}G_{\nu_1 \conv \nu_2}(x+i y)$ also exists and 
from (\ref{inequality}) the estimate follows: 
$$
|G_\omega(x) - G_{\nu_1 \conv \nu_2}(x)| \leq c_6 p \delta^{-3/2},\q x \in [-2 + \delta, 2 - \delta].
$$
Hence we conclude 
$$|p_\omega(x) - p_{\nu_1 \conv \nu_2}(x)| \leq  c_7 p \delta^{-3/2},\q x \in [- 2 + \delta, 2 - \delta].$$
It easy to see  
$p_\omega(x) >\sqrt {\delta}/\pi$ on $[- 2 + \delta,\ 2 - \delta]$.
If we choose $p$ such that   $c_7 p \delta^{-2}< 1/2\pi$, 
then  $p_{\nu_1 \conv \nu_2}(x)> 0$ on $[- 2 + \delta,\ 2 - \delta]$.
Analyticity of $p_{\nu_1 \conv \nu_2}$ follows from Theorem  \ref{Belinschi}.
\end{proof}


\begin{corollary}
\label{an_cont}
For every $\delta \in (0,1/10)$ and $m \geq n \geq c(\mu^*,s) \est$ 
the following measures have a positive and analytic density:
1) $\omega \conv \mu_s^{(\unve_s)}$,
2) $\mu_n$,
3) $\muves$.
Moreover,  Cauchy transforms $G_{\omega \conv \mu_s^{(\unve_s)}}$, 
$G_{\mu^*_n}$ and $G_{\muves}$ extend analytically to a neighbourhood  of $[- 2 + \delta, 2 - \delta]$.
\end{corollary}

\begin{proof}
1) Due to Lemma \ref {ocenka_levy} the following bound holds:
$$
d_L(\omega,\omega \conv \mu_s^{(\unve_s)}) \leq 2s n^{-1/3}.
$$
By Proposition \ref{bla-bla}  the density $p_{\omega \conv \mu_s^{(\unve_s)}}(x)$
is positive and analytic on $[- 2 + \delta, 2 - \delta]$ for $n \geq c(\mu^*,s) \est$. 

2) By the Berry-Esseen inequality (\ref{b-e})
$
d_L(\omega, \mu^*_m) \leq c(\mu^*)/ \sqrt m.
$
By Proposition \ref{bla-bla}  the density $p_{\mu^*_m}(x)$
is positive and analytic on $[- 2 + \delta, 2 - \delta]$, for $m \geq c(\mu^*) \est$.

3) By the Berry-Esseen inequality (\ref{b-e})
$$
d_L(\omega,\muves) \leq \frac {c_1(\mu^*)} {\sqrt m} + \frac {c_2(\mu^*,s)} { n^{1/3}} 
\leq \frac {c_3(\mu^*,s)}{n^{1/3}}.
$$
By Proposition \ref{bla-bla}   the density $p_{\muves}(x)$
is positive and analytic on $[- 2 + \delta, 2 - \delta]$ for $m \geq n \geq c(\mu^*,s) \est$.

Analyticity of the Cauchy transforms follows from Theorem \ref{Belinschi}.
\end{proof}


\subsection{Analytic continuation for $G_{\muves}$.}
Below we prove Theorem \ref{extension_33} which shows that the Cauchy transform $G_{\muves}$ 
has an analytic continuation on
$$K := \{x+iy:x \in [-2 + 2 \delta, 2 - 2\delta]; |y| < \delta \sqrt \delta \}.$$


\begin{theorem}
\label{extension_33}
For every $\delta \in (0,1/10)$ and $m \geq n \geq N(: = c (\mu^*,s) \est$) the Cauchy transform 
$G_{\muves}$ has an analytic continuation on $K$
such that
\begin{equation}
\label{lskjdf}
G_{\muves}(z) = G_{\omega} (z) + \tilde l (z),
\end{equation}
where 
$|\tilde l (z)| \leq \frac {c(s)}{\sqrt {n  \delta}}$ on $K$.
\end{theorem}


\begin{proof}
The inverse function of $G_{\muves}$ can be expressed as
$$
G^{(-1)}_{\muves}(w) 
= \sum_{j=1}^{2s} R_{D_{\ve_j} \mu^*}(w) + 
(m - s) R_{D_{1/\sqrt m}\mu^*}(w)
 + \frac 1 w,
$$
for  $w$, such that the series $R_{D_{\ve_j} \mu^*}(w)$ and 
$(m - s)R_{D_{1/\sqrt m}\mu^*}(w)$
converge.
Due to the rescaling property of the $R$-transform (\ref {dilation})
we have
\begin{eqnarray*}
(m - s)R_{D_{1/\sqrt m}\mu^*}(w) =
w - \frac {s w}{m} + 
(m - s) \left (\sum_{l=2}^{7} \frac {\kappa^*_{l+1} w^l}{m^{(l+1)/2}}
+ \sum_{l=8}^{\infty}\kappa'_{l+1}w^l \right),
\end{eqnarray*}
where $\kappa_l^*$ and $\kappa'_l$ are free cumulants of $\mu^*$ and $D_{1/\sqrt m}\mu^*$ respectively.
For $w \in D_{\theta,1.4}$  (see Lemma \ref {sets}) 
by  inequalities (\ref{cum}) we obtain the estimate:

\begin{eqnarray*}
\left |\sum_{l=2}^{7} \frac {\kappa^*_{l+1} w^l}{m^{(l+1)/2}} +  \sum_{l=8}^{\infty} \kappa'_{l+1}  w^{l} \right|  & \leq & \frac { c(s)} {m^{3/2}}.
\end{eqnarray*}
Hence
$
(m - s)R_{D_{1/\sqrt m}\mu^*}(w) = w + g_1(w),
$
where $|g_1(w)| \leq c_1(s)/\sqrt m$ on $D_{\theta,1.4}$, $m \geq N$.


In the same way we obtain the estimate:

\begin{eqnarray*}
\left| \sum_{j=1}^{2s}  R_{D_{\ve_j}\mu^*}(w)\right|
\leq \frac {c_2(s)} n, \q w \in D_{\theta,1.4}.
\end{eqnarray*}


Due to Lemma \ref{sets} we know  $G_{\omega}(K_{\delta}) \subset D_{\theta,1.4}$. 
Thus   replacing  $w$ by $G_\omega$   the we get in the view of the functional equation (\ref{func}) 
\begin{equation}
\label{lskdjfhg}
f(z) := G^{(-1)}_{\muves}(G_\omega(z)) = z + g(z), \quad z \in K_{\delta},
\end{equation}
where $g(z)$ considered as a power series in $z$ 
$$
g(z) = \sum_{j=1}^{2s}  R_{D_{\ve_j}\mu^*}(G_\omega(z)) + 
(m - s)  R_{D_{1/\sqrt m}\mu}(G_{\omega}(z))
$$
converges uniformly on $K_{\delta}$ to zero as $n \to \infty$ and the estimate 

$$
|g(z)| \leq \frac {c(s)} {\sqrt n} 
$$
 holds uniformly on $K_{\delta}$ for $m \geq n \geq N$.
The uniform bound of $|g(z)|$ and (\ref{lskdjfhg}) imply that the rectangle 
$K$ 
is contained in the set $f(K_{\delta})$. 
Rouch\'{e}'s Theorem 
 implies that each 
function $f$ has an analytic inverse $f^{(-1)}$ defined on $K$.
Due to (\ref{lskdjfhg}) it follows that
$$
z = f\left(f^{(-1)}(z)\right) = f^{(-1)}(z) + g\left(f^{(-1)}(z)\right)
$$
$$
f^{(-1)}(z) = z - \widetilde{g}(z), \quad z \in K, 
$$
where $\widetilde{g}(z) = - g\left(f^{(-1)}(z)\right)$, $f^{(-1)}(z)  \in K_\delta$ 
for $z \in K$, hence 
$$
|\widetilde{g}(z)| \leq  \frac {c(s)} {\sqrt n} 
,\ \  z \in K,\ \ m \geq n \geq N.
$$

By Corollary \ref{an_cont} the function $G_{\muves}$ has an analytic continuation to the interval
$[-2 +  \delta, 2 -   \delta]$ for $m \geq n \geq N$. 
The composition $G_\omega^{(-1)}\circ G_{\muves}$ is defined and  analytic in a neighbourhood  of the interval 
$[-2 + \delta, 2 -  \delta]$ and hence, it coincides with the function 
$f^{(-1)}$ on  $[-2 + 2 \delta,\  2 - 2 \delta]$. 
We conclude
\begin{equation}
\label{invgg}
G_\omega^{(-1)}( G_{\muves}(z)) = f^{(-1)}(z) = z + \widetilde{g}(z), 
\quad z \in K, \quad m \geq n \geq N.
\end{equation}
Let us estimate $|G'_{\omega}(z)|$ on $K$. It is easy to see 
\begin{eqnarray*}
\label{411}
|G'_{\omega}(z)|= \left|\frac 1 2 + \frac {i z} {2 \sqrt {4 - z^2}}\right |
\leq \left | \frac 1 2 + \frac {i(2 - i \delta \sqrt \delta)}{4 \sqrt {2\delta}}\right| \leq \frac 1{2 \sqrt \delta},\q
z \in K.
\end{eqnarray*}
Applying  $G_\omega$ on (\ref{invgg}), we get
\begin{eqnarray*}
G_{\muves}(z) = G_\omega(z + \widetilde{g}(z)) = G_\omega(z) +\tilde l(z),
 \quad z \in K, \q m \geq n \geq N,
\end{eqnarray*}
where 
$$
|\tilde l(z)| \leq \sup_{z \in K} |G^\prime_\omega(z)| |\widetilde{g}(z)|
\leq   \frac {c(s)} {\sqrt {n  \delta}},\ \  z \in K,\ \  m\geq  n\geq N.
$$
Thus the theorem is proved.
\end{proof}

Recall, that in our case $\mu^* = \mu$.
\begin{proof}[Proof of Corollary \ref {extension_g}]
The  statement follows from Theorem \ref{extension_33} with $m=n$ and $s =0$. 
\end{proof}

\begin{proof}[Proof of Corollary \ref {lklklk}]
In Theorem \ref{extension_33} we put $g_1(z) =0$, thus the corollary  is proved.
\end{proof}

\begin{proof}[Proof of Corollary \ref {dfdfdfdfd}]
Combining Corollary \ref {lklklk} and Theorem \ref{extension_33} we obtain the statement.
\end{proof}

\begin{proof}[Proof of Corollary \ref {sym_comp}]
The function
$G^{(-1)}_{\muves}(w)$, $w \in D_{\theta,1.4}$ is symmetric and compatible. 
Hence by (\ref {lskdjfhg}) and (\ref {invgg}) we may conclude that $G_{\muves}(z)$, $z \in K$ is symmetric and compatible.
\end{proof}

\subsection{Proofs of Theorem  \ref{asdfasdf} and Theorem  \ref{vanish_dir}. }

The results obtained so far allow us to prove
 Theorem \ref{asdfasdf}.
\begin{proof}[Proof of Theorem \ref{asdfasdf}]
Let us define the set
$$U_0: = \{\underline \eta_{2s} \in \compl^{2s}:\ |\eta_j| \leq 1/\sqrt n, \ j = 1,\dots, 2s\}$$ and  the function
\begin{eqnarray*}
G^{(-1)}(\underline \eta_{2s},w)  =  
w + \frac 1 w - \frac {s w} m
+ 
(m - s)  \sum_{l=2}^\infty \kappa'_{l+1} w^l +
\sum_{j=1}^{2s} \sqrt n\eta_j \sum_{l=1}^\infty \kappa''_{l+1}(\sqrt n \eta_j w)^l, 
\end{eqnarray*}
where $\kappa'_l$ and $\kappa''_l$ are free cumulants of $D_{1/\sqrt m}\mu^*$ and $D_{1/\sqrt n}\mu^*$ respectively, and
 $ w\in D_{\theta, 1.4}$,  $\underline \eta_{2s} \in U_0$,
such that
$$
G^{(-1)}(\underline \eta_{2s},w)\Big|_{\underline \eta_{2s} = \unve_{2s}} = G^{(-1)}_{\muves}(w).
$$
The function 
$G^{(-1)}(\underline \eta_{2s},w)$ is analytic on $U_0\times D_{\theta, 1.4}$.
Consider  the function
$$
F(\underline \eta_{2s},z,w)  = 
G^{(-1)}(\underline \eta_{2s},w) - z,
$$
for
$w \in D_{\theta, 1.4}$, $z \in G^{(-1)}(\underline \eta_{2s}, D_{\theta,1.4})$ and $\underline \eta_{2s} \in U_0$.

This function is analytic on $ U_0 \times  G^{(-1)}(\underline \eta_{2s}, D_{\theta,1.4}) \times D_{\theta, 1.4}$.
For fixed $\unve^0_{2s} \in \real^{2s} \cap U_0$, $w_0 \in D_{\theta,1.4}$  
and fixed $z_0 = G^{(-1)}(\unve^0_{2s},w_0) \in G^{(-1)}(\unve^0_{2s},D_{\theta,1.4})$
we have
$
F(\unve_{2s}^0,z_0,w_0) =
0 \q \mbox{and}
$
\begin{eqnarray*}
\lefteqn{
\frac {\partial}{\partial w}F(\unve^0_{2s},z_0,w_0)} \\ &\!\!\! = &\!\!
1 - \frac 1 {w^2_0} - \frac {s} m
+  (m - s) 
\sum_{l=2}^\infty l \kappa'_{l+1}w_0^{l-1}  +
\sum_{j=1}^{2s} (\ve^0_j)^2 n \sum_{l=1}^\infty l \kappa''_{l+1}(\ve^0_j \sqrt n w_0)^{l-1}\!. 
\end{eqnarray*}
Using the estimates $|w_0^2 - 1| > \sin^2 \theta > \delta/16$ on $D_{\theta, 1.4}$
and 
$$
\left|(m - s) \sum_{l=2}^\infty l \kappa'_{l+1}w_0^{l-1} +
\sum_{j=1}^{2s} (\ve^0_j)^2 n \sum_{l=1}^\infty l \kappa''_{l+1}(\ve^0_j \sqrt n  w_0)^{l-1} - \frac s m \right|
\leq \frac c {\sqrt n},
$$
we conclude
$$
\left|\frac {\partial}{\partial w}F(\unve^0_{2s},z_0,w_0) \right| > c \delta >0.
$$
Due to the Implicit Function Theorem  \cite{FG02}
for every point $(\unve_{2s}^0,z_0,w_0)$ 
  there is an open neighbourhood 
$U = \tilde U_0 \times U_{z_0}\times U_{w_0} \subset U_0 \times  G^{(-1)}(\unve^0_{2s}, D_{\theta,1.4}) \times D_{\theta, 1.4}$ 
and  an analytic function 
$G: \tilde U_0 \times U_{z_0} \to U_{w_0}$
such that $G(\underline \eta_{2s}, z; \unve_{2s}^0,z_0) = w_0$.
Moreover, 
$$
G(\underline \eta_{2s},z;\unve_{2s}^0, z_0)\Big|_{\underline \eta_{2s} = \unve_{2s}} = G_{\muves}(z),\ \ z_0 \in K \subset G^{(-1)}(\unve^0_{2s}, D_{\theta,1.4}).
$$
Note, that for $z_0^1 \neq z_0^2$, $z \in  U_{z_0^1} \cap U_{z_0^2}$
and $\unve_{2s}^{0,1} \neq \unve_{2s}^{0,2}$, 
$\underline \eta_{2s} \in U_{\unve_{2s}^{0,1}} \cap U_{\unve_{2s}^{0,2}}$
the functions $G(\underline \eta_{2s},z;\unve_{2s}^{0,1},z^1_0)$ and  
$G(\underline \eta_{2s},z;\unve_{2s}^{0,2},z^2_0)$ do not necessarily coincide, 
however
$$
G(\unve_{2s},z;\unve_{2s}^{0,1},z^1_0) = G(\unve_{2s},z;\unve_{2s}^{0,2},z^2_0) = G_{\muves}(z),\q z_0^1,z_0^2 \in K,
$$
since $G_{\muves}(z)$ is uniquely defined for $z \in K$ by Corollary \ref{lklklk}.
We conclude that 
$G_{\muves}(z)$ is real analytic with respect to the $2s$ variables $\ve_j$ such that 
$|\ve_j|  \leq n^{-1/2}$  and complex analytic with respect to $z \in K$
for $m \geq n \geq N$.


Moreover, $|G(\underline \eta_{2s},z,\unve_{2s}^0, z_0)|$ is uniformly bounded in a neighbourhood of 
$(\unve_{2s}^0, z_0)$, $\unve_{2s}^0 \in \real^{2s} \cap U_0$, $z_0 \in K$, $n \geq m \geq N$.
Therefore, $|G_{\muves}(z)|$  is uniformly bounded on $\set \times K$.
\end{proof}

\begin{proof}[Proof of Theorem \ref{vanish_dir}]
Consider the rescaled measures
$$
\tilde \mu_{m - s} :=\underbrace{D_{m^{-1/2}}\mu^*\conv  \dots \conv D_{m^{-1/2}}\mu^*}_{m - s\ times}.
$$
Let us calculate $\frac \partial {\partial \ve_j}G_{\muves}(z)$ at $\ve_j=0$, $j = 1, \dots, 2s$ for  $z\in K$, $m \geq n \geq N$.
For this purpose, we differentiate the
equation 
$$
z = R_{\muves} (G_{\muves}(z)) + \frac 1 {G_{\muves}(z)},
$$
and arrive at
\begin{eqnarray}
\label{kfjdhs}
0& = &\left[ R'_{\tilde \mu_{m - s}} (G_{\muves}(z)) \frac \partial {\partial \ve_j}G_{\muves}(z)
+R_{\mu^*}(\ve_jG_{\muves}(z))\right.\nonumber \\ 
& + & \ve_j R'_{\mu^*}(\ve_j G_{\muves}(z))(G_{\muves}(z) 
 +  \ve_j \frac \partial {\partial \ve_j}G_{\muves}(z))\nonumber \\
&+&
\left.
\sum_{i=1} ^{2s}{^*}\ve^2_i R'_{\mu^*} ( \ve_i G_{\muves}(z))\frac \partial {\partial \ve_j}G_{\muves}(z)
-\frac{\frac \partial {\partial \ve_j}G_{\muves}(z)} {G^2_{\muves}(z)}\right]\Bigg|_{\ve_j =0},\q\q
\end{eqnarray}
where $\sum_{i=1}^{2s}\! {^*}$ means summation  over all  $i \neq j$.
After simple computations we get
\begin{eqnarray*}
0 &= & R'_{\tilde \mu_{m - s}} (G_{\muves} (z)) \frac \partial {\partial \ve_j}G_{\muves}(z)\Big|_{\ve_j =0}
- \frac{\frac \partial {\partial \ve_j}G_{\muves}(z)} {G^2_{\muves}(z)}\Bigg|_{\ve_j  =0}\\
&+&\sum_{i=1}^{2s} {^*}\ve^2_i R'_{\mu^*} ( \ve_i G_{\muves}(z))\frac \partial {\partial \ve_j}G_{\muves}(z)\Big|_{\ve_j  =0},
\end{eqnarray*}
By the definition of the $R$-transform and taking into account that $\mu^*$ has zero mean and unit variance we obtain
\begin{eqnarray*}
 R'_{\tilde\mu_{m - s}}(z) = 
(m - s)R'_{D_{1/\sqrt m}\mu^*} (z)  =
(m - s )\left(\frac 1 m + \sum_{l=2}^\infty l \kappa'_{l+1} z ^{l - 1}\right).
\end{eqnarray*}
Finally, $\frac \partial {\partial \ve_j}G_{\muves}(z)$ satisfies the equation:
\begin{eqnarray}
\label{long_eq}
\lefteqn{
\left[
\left(m - s \right ) \left( \frac 1 m + \sum_{l=2}^\infty l \kappa'_{l+1} \left( G_{\muves}(z) \right)^{l-1}\right)
G^2_{\muves}(z) -1\right.}\\
&+&\left. G^2_{\muves}(z) \sum_{i=1}^{2s} {^*}n \ve^2_i  \sum_{l=2}^\infty l \kappa''_{l+1} \left( \sqrt n\ve_i G_{\muves}(z)  \right)^{l-1})
\right]\frac \partial {\partial \ve_j}G_{\muves}(z)\Bigg|_{\ve_j  =0} = 0. \nonumber
\end{eqnarray}
Using the representation
\begin{eqnarray*}
\label{muves}
G_{\muves}(z) = G_\omega(z) + \tilde l(z),\q z \in K,\ \  m \geq n \geq N
\end{eqnarray*}
where
$
|\tilde l(z)| \leq \frac 1 {\sqrt {n \delta}},$
$ z \in K,$ $ m\geq n \geq N
$
 we rewrite equation (\ref{long_eq}) in the following way
\be*
(G^2_\omega(z) - 1 + f(z))\frac {\partial}{\partial \ve_j}G_{\muves}(z)\Big|_{\ve_j  =0} = 0,
\ee*
where 
$|f(z)| \leq \frac c {\sqrt {n \delta}}$, $ z \in K,$ $ m\geq n \geq N.$
%
%
Finally, we can find an $N^*$  such that for every $m \geq n \geq N$ 
$$
|G_\omega^2(z) - 1| > |f(z)|,\q z \in \partial K,
$$
see (\ref {ergthnj}) below.
By Rouch\'{e}'s theorem we conclude that $G^2_\omega(z) - 1 +  f(z)$ has no roots on $K$, $m \geq n \geq N$,  thus 
$\frac {\partial}{\partial \ve_j}G_{\muves}(z)\Big|_{\ve_j  =0}\! = 0$ for $z \in K$, $m \geq n \geq N$. 
\end{proof}


\subsection{Proofs of Theorem \ref{exp_Cauchy}, Corollary \ref{cor5} and Corollary \ref{cor6}.}

We start by computing  the derivatives of $G_{\mutilda}$.
%
%
%
%
%
The extension  $G_{\omega\conv \mu_r^{(\unve_r)}}$ is
defined by (see  (\ref {cont41}))
\begin{eqnarray*}
z =  \sum_{i=1}^s  R_{D_{\ve_i}\mu^*}( G_{\mutilda}(z)) +
G_{\mutilda}(z)+ \frac{1}{G_{\mutilda}(z)}.
\end{eqnarray*}
In view of  the rescaling property of the $R$-transform we arrive at
\begin {eqnarray*}
z =  \sum_{i=1}^s \varepsilon_i R_{\mu^*}(\varepsilon_i G_{\mutilda}(z)) +
G_{\mutilda}(z)+\frac{1}{G_{\mutilda}(z)}.
\end{eqnarray*}
Below we will use the notation:
$
h_{\infty}(\unve_s;z): = G_{\mutilda}(z).
$
We set 
$$
F(\unve_s,z,h_{\infty}(\unve_s;z)) :=  \sum_{i=1}^s \ve_i R_{\mu}(\ve_i h_{\infty}(\unve_s;z)) +
h_{\infty}(\unve_s;z)+\frac{1}{h_{\infty}(\unve_s;z)}-z.
$$
Using these representations we may determine
 the derivatives of $h_{\infty}(\unve_s,z)$ as solutions of  the equations 
\begin{eqnarray}
\label{eq-der}
D^\alpha F(\unve_s,z,h_{\infty}(\unve_s;z)) \Big|_{\unve_s=0} = 0, \q |\alpha| \leq s.
\end{eqnarray}
Let us compute the first derivative of $h_{\infty}(\ve;z)$ at $\ve=0$, $z \in K$. Setting in (\ref{eq-der}) $\alpha = 1$  we obtain
\begin{eqnarray*}
\frac {\partial} {\partial \ve}F(\ve,z,h_{\infty}(\ve;z))\Big|_{\ve = 0} = 0.
\end{eqnarray*}
After simple computations, we arrive at the  equation:
$$
 \left( 1 - \frac {1}{G^2_{\omega}(z)}\right)\frac {\partial} {\partial \ve}
h_{\infty}(\ve,z)\Big| _{\ve =0} = 0.
$$
Due to Lemma \ref{sets}, $G_{\omega} (K) \subset D_{\theta, 1.4}$, where 
$ 2 \sin \theta = \sqrt {\frac {\delta} {4} \left( 1 - \frac {\delta} {4}\right)}$. Hence
$|G^2_{\omega}(z)| \leq 1 - \delta/16$ and 
\begin{equation}
\label{ergthnj}
|G^2_{\omega}(z) -1| \geq \delta/16 > 0, \q z \in K.
\end{equation}
Thus, we get 
$\frac {\partial} {\partial \ve}h_{\infty}(\ve;z)\Big| _{\ve=0} = 0.$

Setting in (\ref{eq-der}) $\alpha = 3$ we get 
$$
\frac {\partial^3} {\partial \ve^3}F(\ve,z,h_{\infty}(\ve;z))\Big|_{\ve = 0} = 0.
$$
After differentiation  and by the inequality (\ref{ergthnj}) we obtain
%
\begin{eqnarray*}
\frac {\partial^3} {\partial \ve^3}h_{\infty}(\ve;z)  \Big|_{\ve=0}=
\frac  {6 \kappa_3 G^4_{\omega}(z)}{1 - G^2_{\omega}(z) }.
\end{eqnarray*}
Continuing  this scheme we obtain the desired result.

\begin{proof}[Proof of Theorem \ref{exp_Cauchy}]
In order to compute the expansion  for $G_{\mu^*_n}$ we apply Theorem  \ref {th}.
By Corollary \ref{sym_comp}
the extension  $G_{\muves}$ is symmetric  
and compatible, thus conditions (\ref{sym}), (\ref{comp}) hold.
Due to Theorem \ref{asdfasdf} the extension $G_{\muves}$ is infinitely differentiable with respect to $\unve_{2s}$, $z \in K$, $m \geq n \geq N$
and  conditions (\ref{der1}) and (\ref{der2}) hold.
Theorem \ref{vanish_dir} shows that condition (\ref{fd}) holds.
Therefore, we get an expansion together with estimates for the error term based on (\ref{expansiondf}). 
In order to determine the expansion for $G_{\mu^*_n}(z)$, $z \in K$, $n \geq N$ we need to compute the derivatives of  
$G_{\mutilda}(z)$, $z \in K$ at zero
 and plug the result into (\ref{expan_2}). 
Using the derivatives of $G_{\mutilda}(z)$
equation (\ref{expansiondf}) leads to
\begin{eqnarray}
\label{jiueyr}
\lefteqn{
G_{\mu^*_n}(z) 
 =  G_\omega (z)+ \frac {\kappa_3G_\omega ^4(z)} {(1 - G_\omega ^2(z) )\sqrt n}}\nonumber \\
&+& \left(\big(\kappa_4 - \kappa_3^2 \big)\frac {G_\omega (z)^5}{1 - G_\omega ^2(z)} +\kappa_3^2\Big(\frac{G_\omega (z)^7}{(1 - G_\omega (z)^2)^2}+\frac{G_\omega (z)^5}{(1 - G_\omega (z)^2)^3}\Big)\right) \frac 1 n \nonumber  \\
&-&\left(\frac{\kappa_5  G_\omega ^6 (z)}{( G_\omega ^2 (z) - 1)}
+ \frac{ \kappa_3^3 G_\omega ^{10}(z) \left(5 G_\omega ^4(z)  - 15 G_\omega ^2(z) + 12\right)}
{\left(G_\omega ^2 (z) - 1\right)^5}\right. \nonumber \\
& - & \left. \frac{ \kappa_3 \kappa_4 G^8_\omega (z) \left(5 G_\omega ^2(z) - 7 \right)}{\left(G_\omega ^2(z) - 1\right)^3}\right)\frac 1 {n^{3/2}}
+ O\left(\frac 1 {n^{2} }\right)
\end{eqnarray} 
for $z \in K$, $n \geq c(\mu^*) \est$.
\end{proof}

\begin{proof}[Proof of Corollary \ref {cor5}]
In order to determine the expansion for  densities we have to  
substitute the extension $G_{\omega}(z)$  by  formula (\ref{semiext})  on the left-hand side of (\ref{jiueyr})
and get the density using Stieltjes inversion formula (\ref{SIF}) by taking the imaginary part. 
\end{proof}

\begin{proof}[Proof of Corollary \ref{cor6}]
We  integrate the expansion for densities and apply inequalities (\ref{ineq1}) and (\ref{ineq2}). As a result we obtain the desired expansion for $\mu_n$.
\end{proof}

\appendix  
\section {Auxiliary results}

\begin{theorem}[\cite{Zor2}]
\label{uniform}
Consider vector spaces $X$, $Y$ over $\real$ and a sequence $\{f_n\}_n$ of functions $f_n:A \to Y$, $A \subset X$.
If all  functions $f_{n}$ are differentiable on 
$A$ and
the sequence $\{f'_n\}_n$ converges uniformly on $A$, 
and if the sequence $\{f_n\}_n$  converges at one point $x_0 \in A$,
then
$\{f_n\}_n$  converges to $f$ uniformly on A. Moreover, 
$f$ is differentiable  and 
$ f'(x) = \lim_{n \to \infty}f'_n(x)$,  $x \in A$.

\end{theorem}

\begin{theorem}[Newton-Kantorovich, \cite{Kant96}]
\label{NK}
Consider vector spaces $X$, $Y$ over $\compl$ and a functional equation $F(t) =0$, where $F: X \to Y$.
Assume that the conditions hold:
\begin{enumerate}
\item
$F$ is differentiable at $t^0 \in X$,
$\|F'(t^0)^{-1}\|_Y \leq \beta_0.$
\item
$t_0$ solves approximately $F(t)= 0$ with  estimate
$\|F'(t^0)^{-1} F(t^0)\|_Y \leq \eta_0.$
\item
$F''(t)$ is bounded in $B_0$ (see below):
$\|F''(t)\|_Y \leq K_0.$
\item
$\beta_0$, $\eta_0$, $K_0$ satisfy the inequality
$h_0 = \beta_0 \eta_0 K_0 \leq \frac 1 2.$
\end{enumerate}
Then there is the unique root $t^*$ of $F$ in 
$B_0: = \{t\in  X:\|t - t^0\|_X \leq \frac {1 - \sqrt{1 - 2h_0}}{h_0} \eta_0\}.$
\end{theorem}




\section{Proof of the general  scheme for asymptotic expansions}

For the simplicity we will use the following short cut:
$$
h_n(\unve_n) : = h_n(\unve_n;t).
$$

\begin{proof}[Proof of Proposition \ref{pr}]
As before, we denote 
$\unve_m: = (\ve_1, \dots, \ve_m)\in \real^m$, where if not specified otherwise
$\ve_1 = \dots = \ve_m = m^{-1/2}$.
Let us denote $\uns_2: = (\sigma_1, \sigma_2)\in \real^2$ such that $|\sigma_j| \leq n^{-1/2}$, $j=1,2$, $m \geq n >3$.
We will identify $(\unve_m,\uns_2),$ and
$(\unve_m, 0,\uns_2) \in \real^{m+3}$. 
In particular, notice that
\be*
h_{m+3}(\unve_m,0,\uns_2 ) 
= h_{m+2}(\unve_m,\uns_2).
\ee*
We will also use the following notation  
$$
h_m(\unve_{m-k}) : = 
h_m(\unve_{m-k},\underbrace{0,\dots,0}_{k}),\ \ m \geq k >0.
$$
Now we  expand the function $h_{m+3}(\unve_{m+1}, \uns_2)$ at the point $(\unve_m,\uns_2)$ and get
\begin{eqnarray}
\label{start}
\lefteqn{
h_{m+3}(\unve_{m+1},\uns_2)  = 
h_{m+3}(\unve_{m},\uns_2)}  \nonumber \\
& + &
\sum_{|\alpha| \leq 2} \alpha!^{-1} D^{\alpha} h_{m+3}(\unve_m,\uns_2) ((\unve_{m+1},\uns_2) - (\unve_m,\uns_2))^{\alpha} 
+R_3(m), 
\end{eqnarray}
where $R_3(m)$ is a remainder in  the Lagrange form:
\begin{eqnarray}
\label{lagrange}
R_3(m) = \frac 1 {3!} \left (t_1 \frac \partial {\partial \ve_1} + \dots + t_{m+1} \frac \partial {\partial \ve_{m+1}}\right)^3 h_{m+1} (\unve_{m+1} - \theta \underline t_{m+1}),
\end{eqnarray}
where  $t_j = m^{-1/2} - (m+1)^{-1/2}$, $j = 1,\dots,m$, $t_{m+1} = m^{-1/2}$ and 
$0 < \theta< 1$.
We can deduce the estimate for $R_3(m)$ from $|m^{-1/2} - (m+1)^{-1/2}| \leq c m^{-3/2}$ and counting number of terms in (\ref{lagrange}):
\begin{eqnarray}
\label{remainder}
|R_3(m) |
 \leq  c d_3(h,n) m^{-3/2},\quad  m \geq n>s.
\end{eqnarray}
We rewrite (\ref{start}) in the following way:
\begin{eqnarray}
\label{difer}
\lefteqn{
h_{m+3}(\unve_{m},\uns_2) - h_{m+3}(\unve_{m+1},\uns_2) } \\ 
&=&- \sum_{|\alpha| \leq 2} \alpha!^{-1} D^{\alpha} h_{m+3}(\unve_m,\uns_2) ((\unve_{m+1},\uns_2) -
(\unve_m,\uns_2))^{\alpha} - R_3(m). \q\q \nonumber
\end{eqnarray}
The next step is  expanding the derivatives  on the right-hand side and making use of  condition (\ref{fd}).
We start with the  second mixed derivatives in (\ref{difer})
\begin{eqnarray*}
\frac  {\partial}{\partial \ve_j}\frac{\partial}{\partial \ve_k}h_{m+3}(\unve_{m},\uns_2) 
& = & 
\frac  {\partial}{\partial \ve_j}\frac{\partial}{\partial \ve_k}h_{m+3}(\unve_{m},\uns_2)\Big|_{\ve_j=\ve_k=0}+
O(d_3(h,n)m^{-1/2})\\
& = & O(d_3(h,n)m^{-1/2}), \quad j\neq k.
\end{eqnarray*}
The other derivatives in (\ref{difer}) have the expansions
\begin{eqnarray*}
\frac  {\partial}{\partial \ve_j}h_{m+3}(\unve_{m},\uns_2)
& =&
\frac  {\partial^2}{\partial \ve_j^2}h_{m+3}(\unve_{m},\uns_2)\Big|_{\ve_j = 0} m^{-1/2}+ O (d_3(h,n)m^{-1}),\\
\end{eqnarray*}
\begin{eqnarray*}
\frac  {\partial^2}{\partial \ve_j^2}h_{m+3}(\unve_{m},\uns_2)
& =&
\frac  {\partial^2}{\partial \ve_j^2}h_{m+3}(\unve_{m},\uns_2)\Big|_{\ve_j = 0}+ O (d_3(h,n)m^{-1/2}).
\end{eqnarray*}
Replacing the derivatives in  (\ref{fd}) by their expansions we obtain
\begin{eqnarray*}
\lefteqn{
h_{m+3}(\unve_{m},\uns_2) - h_{m+3}(\unve_{m+1},\uns_2)}  
\nonumber \\
& =  &\sum_{j=1}^m \frac  {\partial^2}{\partial \ve_j^2}h_{m+3}(\unve_{m},\uns_2)\Big|_{\ve_j = 0}
\left[\frac 1 2  (m^{-1} - (m+1)^{-1})\right]\\ 
\nonumber
&  - &\frac 1 2(m  +1) ^ {-1}
 \frac  {\partial^2}{\partial \ve_{m + 1}^2}h_{m + 3}(\unve_m,\uns_2)\Big|_{\ve_{m+1} = 0}
+O\left(d_3(h,n)m^{-3/2}\right).\\ \nonumber
\end{eqnarray*}
Since the function $h_{m+3}(\cdot)$ is symmetric we arrive at
\begin{eqnarray}
\label{kjoiuy}
\lefteqn{
h_{m+3}(\unve_m,\uns_2) - h_{m+3}(\unve_{m+1},\uns_2)}\\
& = & \frac 1 {2(m  +1)} \left(
 \frac  {\partial^2}{\partial \ve_{1}^2}h_{m + 3}(\unve_m,\uns_2)\Big|_{\ve_{1} = 0}
 -    \frac  {\partial^2}{\partial \ve_{m + 1}^2}h_{m + 3}(\unve_m,\uns_2)\Big|_{\ve_{m+1} = 0} \right) \nonumber \\
&& + \ \  O\left(d_3(h,n)m^{-3/2}\right).\nonumber 
\end{eqnarray}
In order to eliminate zero at the $(m+1)$st place of 
$\frac  {\partial^2}{\partial \ve_{j}^2}h_{m + 3}(\unve_m,\uns_2)\Big|_{\ve_{j} = 0}$, $j = 1,\dots, m$
we apply the Taylor series in the following way:
\begin{eqnarray}
\label{zmz}
\frac  {\partial^2}{\partial \ve_j^2}h_{m+3}(\unve_m,\uns_2)\Big|_{\ve_j = 0}
 = 
\frac  {\partial^2}{\partial \ve_j^2}h_{m+3}(\unve_m,\uns_2)\Big|_{\ve_j = 0,\ \ve_{m+1} = m^{-1/2}}
+O\left(d_3(h,n)m^{-1/2}\right). 
\end{eqnarray}
Pluging (\ref{zmz}) into (\ref{kjoiuy}) and using the symmetry condition we conclude
$$
h_{m+3}(\unve_m,\uns_2) - h_{m+3}(\unve_{m+1},\uns_2) = O\left(d_3(h,n)m^{-3/2}\right).
$$
It is easy to see that
$$
h_{m+k+2}(\unve_{m+k},\uns_2) - h_{m+k+3}(\unve_{m+k+1},\uns_2) = O\left(d_3(h,n)(m+k)^{-3/2}\right).
$$
Summing up these differences for $r\geq m$, we obtain
$$
\sum_{k=0}^{r-1}(h_{m+k+2}(\unve_{m+k},\uns_2) - h_{m+k+3}(\unve_{m+k+1},\uns_2)) =
O\left(d_3(h,n)\right)\sum_{k=0}^{r-1}(m+k)^{-3/2}.
$$
Hence,
\begin{equation}
\label{seq}
h_{m+2}(\unve_{m},\uns_2) - h_{m+r+2}(\unve_{m+r},\uns_2) = O\left(d_3(h,n)\right)\sum_{k=0}^{r-1}(m+k)^{-3/2}.
\end{equation}
Finally, (\ref{seq}) shows that $h_{m+2}(\unve_m,\uns_2)$, $m = n, n+1, \dots$  is a Cauchy sequence in $m$ with a limit which we  denote by $h_{\infty}(\uns_2)$, $|\sigma_j| \leq n^{-1/2}$, $j=1,2$. 
Taking $m=n$  and letting $r\to \infty$ in (\ref{seq}) we obtain 
$$
h_{n+2}(n^{-1/2}, \dots, n^{-1/2}, \uns_2) - h_{\infty}(\uns_2) = O\left(d_3(h,n)n^{-1/2}\right),
$$
which proves the proposition.
\end{proof}

The following lemma describes the procedure of eliminating zeros like the one that is used in (\ref{zmz}). 
The lemma shows that additional variables can be introduced (according to the compatibility property of $h_m$). Then we can differentiate with respect to the additional variables at zero instead of differentiating with respect to 
$\ve_j$, $j =1,...,m+1$.

\begin{lemma}
\label{lemma}
Suppose that conditions $(\ref{sym})-(\ref{fd})$ hold. 
Then
\begin{eqnarray}
\label{lemma1}
\lefteqn{
\sum_{j=1}^k \frac {\partial^j} {\partial \ve^j} h_{m+1} (\ve, \ve_2, \dots, \ve_{m+1})\Big |_{\ve = 0} j!^{-1}
(\eta^j - \ve^j)} \\ 
&=& \sum_{r=1} ^k \widetilde{P}_r\left((\eta^. - \ve^.)\kappa_.(D)\right)
h_{m+1+k}(\lambda_1,\dots, \lambda_k, \ve, \ve_2, \dots, \ve_{m+1})
\Big|_{\underline\lambda_k = 0} \nonumber \\
& + & O(m^{-(k+1)/2}), \nonumber
\end{eqnarray}
where the differential operators $\widetilde{P}_r$ and $\kappa_p$ are defined in $(\ref{polynomial})$ below and $(\ref{ln})$, and
$$
(\eta^. - \ve^.)\kappa_.(D): = ((\eta^p - \ve^p)\kappa_p(D),\ \ p = 1, \dots, r).
$$
\end{lemma}

\begin{proof}
The differential operators
$\widetilde{P}_r(\tau_.\kappa_.)$ are polynomials in the cumulant operators $\kappa_p$ (see (\ref{ln})) 
multiplied by formal variables $\tau_p$, $p = 1, \dots, r$.
These polynomials are defined by the formal power series in $\tau_p$
\begin{equation}
\label{polynomial}
\sum_{j = 0}^{\infty} \widetilde{P}_j(\tau_.\kappa_.(D))\mu^j = 
\exp\left(\sum_{j = 2}^{\infty} j!^{-1} \tau_j \kappa_j(D)\mu^j\right).
\end{equation}
When  $\tau_j = \tau^j,\ j\geq1$, then due to (\ref{ln}) we have
\begin{eqnarray*}
\sum_{j=0}^{\infty}\widetilde{P}_j(\tau_.\kappa_.(D)) =
1 + \sum_{j=2}^\infty j!^{-1} \tau_j D^j.
\end{eqnarray*}
 Hence, $\widetilde{P}_0(\tau_.\kappa_.(D)) = 1$, $\widetilde{P}_1(\tau_.\kappa_.(D)) = 0$ and
$\widetilde{P}_j(\tau_.\kappa_.(D)) = j!^{-1} \tau_j D^j$, $j\geq 2$, 
which means that the differential operators $\widetilde P_r$ are nothing else than derivatives of order $r$ multiplied by
 $r!^{-1}$ and the corresponding power of the formal variable $\tau_r$. 
It easy to see that $\widetilde P_r$ gives the $r$th term in the Taylor expansion so that we can write
\begin{equation*}
h_m(\unve_m) = \sum_{j=0}^{k} \widetilde{P}_j(\ve^._i\kappa_.(D))h_m(\unve_m)\Big | _{\ve_i = 0} + O(m^{-(k+1)/2}),
\quad  i = 1,\dots, m.
\end{equation*}
Notice that  $\widetilde P_r$ depends on the cumulant differential operators $\kappa_.(D)$. 
These operators consist of derivatives with respect to multi-variables, 
for instance $\kappa_4(D) = D^4 - 3 D^2D^2$. 
Here $D^2D^2$ denotes differentiation with respect to two different variables (we do not need to specify the variables because of the symmetry condition). 
Therefore, we  introduce additional variables,  say $\underline\lambda_k$, and  write
$$
h_m(\unve_m) = \sum_{j=0}^{k} \widetilde{P}_j(\ve^._i\kappa_.(D))h_{m+k}
(\underline{\lambda}_k,\unve_m)\Big | _{\underline\lambda_k=\ve_i = 0} + O(m^{-(k+1)/2}),
$$
$ i = 1,\dots, m$.
The advantage of the operators $\widetilde P_r$ is that they are defined by exponents which can be easily  reordered by  the properties of exponential functions. 
Due to  (\ref{polynomial}) and the multiplication theorem for exponential functions we obtain
\begin{eqnarray*}
\sum_{j+l=r}\widetilde{P}_j(\tau_.\kappa_.)\widetilde{P}_l(\tau_.'\kappa_.) = 
\widetilde{P}_r((\tau_. + \tau_.')\kappa_.) \quad (\tau_. = (\tau_1, \dots, \tau_{r})).
\end{eqnarray*}
In order to prove the theorem we start from the right-hand side of (\ref{lemma1}): 
\begin{eqnarray*}
\lefteqn{
\sum_{r=1}^k \widetilde{P}_r ((\eta^. - \ve^.)\kappa_.(D))h_{m+1+k}(\lambda_1,\dots\lambda_k,\ve,\ve_2, \dots, \ve_{m+1}) \Big|_{\lambda_1 = \dots = \lambda_k = 0}}\\
&=& \sum_{r=1}^k \widetilde{P}_r ((\eta^. - \ve^.)\kappa_.(D))  \sum_{l=0}^ {k-r}\widetilde{P}_l (\ve^. \kappa_.(D))\\
&\times&h_{m+1+k}(\lambda_1, \dots, \lambda_{k}, \ve,\ve_2,\dots, \ve_{m+1})\Big|_{\lambda_1 = \dots = \lambda_{k} = \ve = 0}+O(m^{-(k+1)/2})\\
&=&\sum_{j=1}^k \sum_{\genfrac{}{}{0pt}{}{l+r = j} {r \geq 1}}\widetilde{P}_r ((\eta^. - \ve^.)\kappa_.(D))\widetilde{P}_l (\ve^. \kappa_.(D))\\
&\times&h_{m+1+k}(\lambda_1, \dots, \lambda_j, \ve,\ve_2,\dots, \ve_{m+1})\Big|_{\lambda_1 = \dots = \lambda_k = \ve = 0}+O(m^{-(k+1)/2})\\
&=&\sum_{j=1}^k \left(\widetilde{P}_j (\eta^. \kappa_.(D) - \widetilde{P}_j (\ve^. \kappa_.(D)\right)
h_{m+1+k}(\underline \lambda_k, \ve,\ve_2,\dots, \ve_{m+1})\Big|_{\underline \lambda_k =0, \ve = 0} \\
&+&O(m^{-(k+1)/2})\\
&=& \sum_{j=1}^k \frac {\partial^j}{\partial \ve^j}h_{m+1}(\ve, \ve_2,\dots, \ve_{m+1})\Big|_{\ve=0}
j!^{-1}(\eta^j - \ve^j) + O(m^{-(k+1)/2}).
\end{eqnarray*}
The last expression coincides with the left-hand side in (\ref{lemma1}), thus the theorem is proved.
\end{proof}

\begin{proof}[Proof of Theorem \ref{th}]

The theorem  will be proved by induction on the length of  the expansion, starting with  $s=4$. 
The case $s=3$ was shown in Proposition \ref{pr}.
Assume that $m \geq n$ ($n \geq 1)$.
We start with the expansion
\begin{eqnarray}
\label{beg}
\lefteqn{
h_{m+1}(\unve_m) - h_{m+1}(\unve_{m+1})} \nonumber\\
&=& - \sum_{0 < |\alpha| < s} \alpha!^{-1} D^{\alpha} h_{m+1}(\unve_m) (\unve_{m+1} - \unve_m)^{\alpha} + R_s(m),
\end{eqnarray}
where  
\begin{equation}
\label{ineq}
|R_s(m)| \leq C d_s (h,n) m^{-s/2}.
\end{equation} 
The last inequality is similar to inequality (\ref{remainder}) in the proof of  Proposition \ref{pr}.

In order to apply  condition (\ref{fd}) on the first derivatives we expand 
$D^{\alpha}h_{m+1} (\unve_m)$, $\alpha = (\alpha_{j_1}, \dots, \alpha_{j_p})$,
$1 \leq j_1 < \dots < j_p \leq m+1$, around $\ve_{j_r} = 0,\ r = 1,  \dots, p$. This yields
\begin{equation}
\label{111}
D^\alpha h_{m+1}(\unve_m)\  = \sum_{0<|\alpha| + |\beta| <s}D^{\alpha + \beta} h_{m+1} (\unve^*_m)
\unve_m^\beta \beta!^{-1} + \widetilde {R}_s(m),
\end{equation}
where $ \widetilde {R}_s(m)$ satisfies  inequality (\ref{ineq}), $\unve^*_m$ is equal to $\unve_m$
except for the components $\ve_{j_1}, \dots, \ve_{j_p}$, which are zero, and $\beta$ is a vector of partial derivatives in the components $j_1, \dots, j_p$. 
Rewrite the derivatives in (\ref{beg}) by their expressions from (\ref{111})
\begin{eqnarray*}
\lefteqn{
h_{m+1}(\unve_{m}) - h_{m+1}(\unve_{m+1})}\\
&=&-\sum_{0<|\alpha|+|\beta| < s} \alpha!^{-1} \beta!^{-1} D^{\alpha + \beta} h_{m+1} (\unve^*_m)
 (\unve_{m+1} - \unve_m)^{\alpha}\unve_m^\beta  +\widetilde{\widetilde R}_s(m),
\end{eqnarray*}
where $\widetilde{\widetilde R}_s(m)$ denotes a remainder term satisfying (\ref{remainder}).

Let $\ve_{m,j} = m^{-1/2}$ and $\ve_{m+1,j} = (m+1)^{-1/2}$, $j=1,\dots m+1$, 
but $\ve_{m, m+1} = 0$.
Using the following relation
\begin{eqnarray*}
\sum_{\genfrac {}{}{0pt}{}{j+k = r}  {j \geq 1}} j!^{-1} k!^{-1} (\ve - \eta)^j \eta ^k = r!^{-1}(\ve^r -\eta^r), 
\quad r\geq1,\ k \geq 0,
\end{eqnarray*}
then we  obtain
\begin{eqnarray}
\label{222}
\!\!\!\! h_{m+1}(\unve_m) - h_{m+1}(\unve_{m+1}) 
 =  -\sum_{0<|\gamma| < s}\!\! \!\gamma!^{-1} D^{\gamma} h_{m+1} (\unve^*_m)
\prod_{j=1}^{m+1}\!\!{^*}(\ve_{m+1,j}^{\gamma_j} - \ve_{m,j}^{\gamma_j})  +
\widetilde{\widetilde R}_s(m), 
\end{eqnarray}
where $\gamma = (\gamma_1, \dots, \gamma_{m+1})$, $\prod^*$  denotes multiplication 
over all $\gamma_j > 0$, $j = 1,\dots,m+1$.

The next step is replacing $\unve^*_m$ by $\unve_m$ in (\ref{222}).
For this purpose we apply Lemma \ref{lemma} to each partial derivative $\gamma_j >0$. 
More precisely, we will take further derivatives with respect to additional variables at zero and make use of the  symmetry condition.
Introduce the  notation
$$
\Delta^._{m,j}: = (\ve^p_{m+1,j }- \ve^p_{m,j}, p = 1,\dots, s - 1).
$$
Applying  Lemma \ref{lemma} to  the derivatives in (\ref{222}) we arrive at
\begin{eqnarray}
\label{step}
\lefteqn{
h_{m+1}(\unve_m) - h_{m+1}(\unve_{m+1})}  \\
& = &\!\!\! - \sum_{k=1}^{m+1} \sum_{(r)}{^*}\widetilde{P}_{r_1}\left( \Delta^._{m,j_1}\kappa_.\right)
\dots \widetilde{P}_{r_k}\left( \Delta^._{m,j_k}\kappa_.\right)
h_{m+1+r}(\unve_m,0, \dots, 0)
 +\ \   R_{1,s}(m), \nonumber
\end{eqnarray}
where $\sum_{(r)}^*$ means summation over all combinations of $r_1, \dots, r_k \geq1$, $k = 1, \dots, m+1$, such that 
$r = r_1 + \dots +r_k <s$ and all ordered $k$-tuples $(j_1, \dots, j_k)$ of indices $1 \leq j_r \leq m+1$ without repetition and $\kappa_. : = \kappa_.(D)$ is a short notation.
Note that the derivatives on the right-hand side of (\ref{step}) define due to conditions (\ref{der1}) and (\ref{der2}).
The remainder term $R_{1,s}(m)$ satisfies (\ref{remainder}).
It easy to see that such a procedure changes nothing for the $(m+1)$st component because the derivatives 
$\frac {\partial^j} {\partial \ve_{m+1}^j} h_{m+1}(\unve_m)$ are expanded at the same  point $\unve_m$.
Relation (\ref{step}) serves as the induction step in the induction on the length of the expansion, say $l$.

Assume that conditions (\ref{sym}) - (\ref{fd}) and (\ref{der1}) - (\ref{der2}) hold with $(s+q)$ instead of $s$. 
Assume we have already proved that for $l = 3, \dots, s - 1$, $m \geq n$, and $| \alpha| \leq s+q$ we have
\begin{eqnarray}
\label{lkasdjf}
\lefteqn{
D^\alpha h_{m+r}(m^{-1/2}, \dots, m^{-1/2}, \ve_1, \dots, \ve_r)\Big|_{\ve_1 = \dots = \ve_r = 0}}\\ 
&=&\sum_{j = 0} ^{l - 3}m^{-j/2} P_j(\kappa_.(D))D^\alpha h_\infty(\underline \lambda_l,\underline \ve_r)\Big|_{\lambda_1 = \dots = \lambda_l = \ve_1= \dots = \ve_r = 0}
+R_{2,l}(m), \nonumber
\end{eqnarray}
where $R_{2,l}(m)$ satisfies
\begin{equation}
\label{444}
|R_{2,l}(m)| \leq c(s) B m^{-(l - 2)/2}.
\end{equation}
The case $l = 3$ follows from Proposition \ref{pr}, where
$$
h_m (\cdot) = D^\alpha h_{m+r}(\cdot, \ve_1, \dots, \ve_r)\Big|_{\ve_1 = \dots = \ve_r = 0},
$$
which  satisfies conditions $(\ref{sym}) - (\ref{fd})$ and $d_3(h,n)<\infty$.

In order to prove (\ref{lkasdjf}) for $l =s$, observe that (\ref{step}) starts with $m+1$ terms 
of order $O(m^{-3/2})$. The induction assumption (\ref{lkasdjf}) with $|\alpha| = 0$ applied to the terms of (\ref{step}) yields
\begin{eqnarray}
\label{123}
h_{m+1}(\unve_m) - h_{m+1}(\unve_{m+1}) 
 = & -& \sum_{k=1}^{m+1} \sum_{(r)}{^{**}}\widetilde{P}_{r_1}\left( \Delta^._{m,j_1}\kappa_.\right)
\dots \widetilde{P}_{r_k}\left( \Delta^._{m,j_k}\kappa_.\right)
m^{-r_0/2} P_{r_0}(\kappa_.) \nonumber\\
&\times&h_{\infty}(\lambda_1, \dots, \lambda_{r_0},\ve_1, \dots, \ve_r)\Big|_{\underline \lambda = \unve = 0}
 + R_{3,s}(m), 
\end{eqnarray}
where $R_{3,s}(m)$ satisfies (\ref{444}) with $l = s+2$, and $\sum_{(r)}^{**}$ denotes summation over all indices $r_1, \dots, r_k \geq 1$, $r_0 \geq 0$ such that $r_0 + \dots + r_k < s$ and all ordered $k$-tuples 
$(j_1, \dots, j_k)$ of indices 
without repetition.

By definition (\ref{polynomial}) of $\widetilde P_r$, the following formal identity holds:
\begin{equation}
\label{234}
\sum_{j = 1}^{\infty} \widetilde{P}_j((\eta^. - \ve^.)\kappa_.) = 
\exp\left(\sum_{j = 2}^{\infty} j!^{-1} (\eta^j - \ve^j) \kappa_j\right) - 1.
\end{equation}
In order to apply this identity to (\ref{123})  we need to change the order of summation in (\ref{123})
in the following way
\begin{eqnarray}
\label{345}
\lefteqn{
h_{m+1}(\unve_m) - h_{m+1}(\unve_{m+1})}  \\
\noindent
& = & - \sum_{r_0 = 0}^{s - 4}m^{-r_0/2} P_{r_0}(\kappa_.)
 \sum_{k=1}^{m+1} \sum_{(j)}{\!^*}  \left[ \prod_{l = 1} ^{s - r_0} \left\{\sum_{v_l = 1} ^\infty 
\widetilde P_{v_l}\left( \Delta^._{m,j_l}\kappa_.\right) \right\}\right]_{s - r_0}\!\!\! h_\infty 
 +  R_{3,s}(m), \nonumber
\end{eqnarray}
where $h_\infty : = h_{\infty}(\lambda_1, \dots, \lambda_{r_0},\ve_1, \dots, \ve_r)\Big|_{\underline \lambda = \unve = 0}$, 
$[\,\  ]_r$ denotes all terms of the enclosed formal power series which are proportional to monomials 
$ \Delta^{p_1}_{m,j_1}$ $\dots$  $\Delta^{p_k}_{m,j_k}$ with $p_1 + \dots + p_k < r$, $k \leq m+1$, and
$\sum_{(j)}^*$ denotes summation over all ordered $k$-tuples $(j_1, \dots, j_k)$ 
 without repetition of the indices.
Applying (\ref{234}) to (\ref{345}), we get
\begin{eqnarray}
\label{456}
\lefteqn{
h_{m+1}(\unve_m) - h_{m+1}(\unve_{m+1})  =  - \sum_{r_0 = 0}^{s - 4}m^{-r_0/2} P_{r_0}(\kappa_.)}\\ 
& \times &
\left[  \sum_{k=1}^{m+1} \sum_{(j)}{\!^*} \prod_{l = 1} ^{s - r_0} \left\{\exp \left[\sum_{p = 2} ^ \infty 
\Delta^p_{m,j_l}p!^{-1}\kappa_p\right] -1 \right\}\right]_{s - r_0}\!\!\!\! h_\infty 
 + R_{3,s}(m). \nonumber
\end{eqnarray}

The identity 
$ \sum_{k=1} ^ {m+1} \sum_{(j)}\!\!^* \prod_{r =1}^k (e_{j_r} - 1) = \prod_{l = 1}^{m+1} e_{j_l} - 1$
together with the symmetry condition of $h_m(\cdot),\ m \geq 1$, shows that (\ref{456}) is equal to
\begin{eqnarray}
\label{exp}
\lefteqn{
h_{m+1}(\unve_m) - h_{m+1}(\unve_{m+1}) }\\ 
& = & -\sum_{r_0 = 0}^{s - 4}m^{-r_0/2} P_{r_0}(\kappa_.)
\left[\exp\left[\sum_{p = 2}^{\infty} \left( \sum_{k = 1} ^ {m+1}\Delta^p_{m,k}\right)p!^{-1} \kappa_p \right] - 1 \right]_{s - r_0} \!\!\! h_\infty  
+  R_{4,s}(m). \nonumber 
\end{eqnarray}
It is easy to see that
\begin{eqnarray}
\label{213}
\sum_{k=1}^{m+1} \Delta^2_{m,k} = \sum_{k=1}^{m+1}\frac 1 {m+1} - \sum_{k = 1}^m \frac 1 m = 0 
\end{eqnarray}
(``equality of variances") and
\begin{equation}
\label{812}
\sum_{k=1}^{m+1} \Delta_{m,k}^p = O(m^{-p/2}), \quad p \geq 3.
\end{equation}
Due to  relation (\ref{213}) the terms for $p=2$ in (\ref{456}) cancel.


By the definition of $P_r$ and $\widetilde P_r$ (see (\ref{134}) and (\ref{polynomial})) it follows that 
\begin{equation}
\label{856}
\sum_{r=1}^\infty \left[P_r(\tau_.\kappa_.) \right]_l = \sum_{r=1}^l \widetilde P_r(\tau_. \kappa_.),
\end{equation}
where, according to the definitions, on the left-hand side $\tau_. = (\tau_3, \dots, \tau_{r+2})$ 
and on the right-hand side $\tau_. = (\tau_3, \dots, \tau_r)$,
and $[\ \ ]_l$ denotes the sum of all monomials $\tau_3^{p_3}\dots \tau_{r+2}^{p_{r+2}}$ in $P_r(\tau_. \kappa_.)$ such that $3 p_3 +  \dots +(r+2) p_{r+2} \leq l$, $l \geq3$.

Applying (\ref{234}) and (\ref{856}) we turn to $P_r$ in (\ref{exp}) and get
\begin{eqnarray}
\label{lskd}
\lefteqn{
m^{-r_0/2} P_{r_0}(\kappa_.)\left[ \exp \left(\sum_{p=3}^\infty \left(\sum_{k=1}^{m+1} \Delta_{m,k}^p\right)p!^{-1} \kappa_p\right) - 1\right]_{s-r_0}h_\infty }\\ 
&=& m^{-r_0/2} P_{r_0}(\kappa_.)  \sum_{r=3}^{s - r_0-1} \widetilde P_r\left (\left( \sum_{k=1}^{m+1} \Delta_{m,k}^.\right)\kappa_.\right) h_\infty \q\q\q\q\q \nonumber \\
&=&\ m^{-r_0/2} P_{r_0}(\kappa_.) \sum_{r=1}^{\infty} \left[ P_r \left( \sum_{k=1}^{m+1} \Delta_{m,k}^. \kappa_.\right)\right]_{s - r_0 - 1}h_\infty. \nonumber
\end{eqnarray}
Finally, (\ref{812}) together with condition (\ref{der2}) shows that
\begin{eqnarray}
\label{876}
\lefteqn{
m^{-r_0/2}P_{r_0} (\kappa_.)P_r \left( \sum_{k=1}^{m+1} 
\Delta_{m,k}^. \kappa_.\right)h_\infty }\\
& = & m^{-r_0/2}P_{r_0} (\kappa_.) \left[ P_r \left( \sum_{k=1}^{m+1} 
\Delta_{m,k}^. \kappa_.  \right)\right]_{s - r_0 - 1} h_\infty + R_{5,s}(m), \nonumber
\end{eqnarray}
where 
\begin{equation}
\label{rem}
|R_{5,s}(m)|\leq Bm^{- s/2}\quad \mbox{for every } m \geq n.
\end{equation}
Note that by  definition (\ref{ppp}), the partial derivatives $D^{(\alpha_1, \dots, \alpha_p)}$ of $h_\infty$ 
on the  right-hand side of (\ref{876}) are such that $\alpha_j \geq 2,\ j =1, \dots, p, \ p \leq k$, and  
$\sum_{j =1}^p (\alpha_j -2) \leq s - 3$.
Relations (\ref{812}), (\ref{lskd}) and (\ref{876}) show that (\ref{exp}) is equal to
\begin{equation}
\label{wry}
- \sum_{r_0 = 0}^{s - 4} m^{-r_0/2} P_{r_0}(\kappa_.)
\sum_{r = 1}^{s - r_0 - 3}P_r \left( \sum_{k=1}^{m+1} 
\Delta_{m,k}^. \kappa_.\right)h_\infty  + R_{6,s}(m),
\end{equation}
where $R_{6,s}(m)$ satisfies (\ref{rem}). Changing the order of summation and applying the  relation
$$
m^{-r_0/2} P_{r_0}(\kappa_.) = P_{r_0} \left (\sum_{j =1}^m \ve^._{m,j} \kappa_. \right),
$$
we obtain that (\ref{wry}) is equal to
\begin{eqnarray*}
\lefteqn{
- \sum_{l = 1}^{s - 3}\sum_{\genfrac {}{}{0pt}{}{r_0+ r = l}{ r \geq 1}}\left[P_{r_0} \left (\sum_{j =1}^m \ve^._{m,j} \kappa_. \right) P_r \left( \sum_{k=1}^{m+1} 
\Delta_{m,k}^. \kappa_.\right) \right] h_\infty + R_{6,s}(m) }\q\q\q \nonumber\\
&=& - \sum_{l = 0}^{s - 3}\sum_{r_0+ r = l }\left[P_{r_0} \left (\sum_{j =1}^m \ve^._{m,j} \kappa_. \right) P_r \left( \sum_{k=1}^{m+1} 
\Delta_{m,k}^. \kappa_.\right) \right] h_\infty   \nonumber\\
& - & \sum_{r_0 = 0}^{s - 3} P_{r_0}\left (\sum_{j =1}^m \ve^._{m,j} \kappa_. \right)
h_\infty + R_{6,s}(m). 
\end{eqnarray*}
By the multiplication theorem for exponential functions
\begin{eqnarray*}
\sum_{r+q=k} P_r(\tau_. \kappa_.) P_q(\tau'_. \kappa_.) = P_k ((\tau_. + \tau'_.)\kappa_.), \quad q,r,k \geq 0,
\end{eqnarray*}
we obtain
\begin{eqnarray*}
\lefteqn{
h_{m+1}(\unve_m) - h_{m+1}(\unve_{m+1}) } \nonumber\\
& = & - \sum_{l = 0}^{s - 3}\left[ P_l \left (\sum_{j =1}^m \ve^._{m,j} \kappa_. + 
\sum_{j = 1}^{m+1} \Delta^._{m,j} \kappa_.\right) 
- P_{l} \left ( \sum_{j=1}^m \ve^._{m,j} \kappa_.\right)\right] h_\infty  
 +  R_{6,s}(m) \nonumber\\
& = & - \sum_{l = 1}^{s - 3}\left[ P_l \left (\sum_{j =1}^{m +1}\ve^._{m+1,j} \kappa_. \right)
- P_{l} \left ( \sum_{j=1}^m \ve^._{m,j} \kappa_.\right)\right ]\! h_\infty \! + R_{6,s}(m). \nonumber \\
\end{eqnarray*}
This implies
\begin{eqnarray*}
h_{m}(\unve_{m}) - h_\infty  (0)
& = & \sum_{k =m}^\infty \left [h_k(\unve_k) - h_{k+1}(\unve_{k+1})\right]\\
& = & \sum_{k =m}^\infty \left[ \sum_{l =1} ^{s - 3}(k^{-l/2} - (k+1)^{-l/2}) P_l(\kappa_.) h_\infty 
+ R_{6,s}(k) \right] \\
& = & \sum_{l =1} ^{s - 3} m^{-l/2}P_l(\kappa_.) h_\infty  + R_{7,s} (m),
\end{eqnarray*}
with $|R_{7,s} (m)| \leq c(s) B m ^{- (s - 2)/ 2}$, where $c(s) > 0$ is  a constant depending on $s$. 
This proves (\ref{lkasdjf}) for $ l = s$ and $|\alpha| = 0$. The case $|\alpha|  > 0$ can be proved similarly. 
Hence, the induction is completed and the theorem is proved.
\end{proof}

\bibliographystyle{acm}

\bibliography{bibliography}

\begin{thebibliography}{10}

\bibitem{Bel08}
{\sc Belinschi, S.~T.}
\newblock The {L}ebesgue decomposition of the free additive convolution of two
  probability distributions.
\newblock {\em Probab. Theory Related Fields 142}, 1-2 (2008), 125--150.

\bibitem{BB07}
{\sc Belinschi, S.~T., and Bercovici, H.}
\newblock A new approach to subordination results in free probability.
\newblock {\em J. Anal. Math. 101\/} (2007), 357--365.

\bibitem{BP96}
{\sc Bercovici, H., and Pata, V.}
\newblock The law of large numbers for free identically distributed random
  variables.
\newblock {\em Ann. Probab. 24}, 1 (1996), 453--465.

\bibitem{BV93}
{\sc Bercovici, H., and Voiculescu, D.}
\newblock Free convolution of measures with unbounded support.
\newblock {\em Indiana Univ. Math. J. 42}, 3 (1993), 733--773.

\bibitem{CHG08}
{\sc Chistyakov, G.~P., and G{\"o}tze, F.}
\newblock Limit theorems in free probability theory. {I}.
\newblock {\em Ann. Probab. 36}, 1 (2008), 54--90.

\bibitem{ChG11}
{\sc Chistyakov, G.~P., and G{\"o}tze, F.}
\newblock The arithmetic of distributions in free probability theory.
\newblock {\em Cent. Eur. J. Math. 9}, 5 (2011), 997--1050.

\bibitem{ChG11v2}
{\sc Chistyakov, G.~P., and G{\"o}tze, F.}
\newblock Asymptotic expansions in the central limit theorem in free
  probability.
\newblock {\em arXiv:1110.4844 v2\/} (2011).

\bibitem{FG02}
{\sc Fritzsche, K., and Grauert, H.}
\newblock {\em From holomorphic functions to complex manifolds}, vol.~213 of
  {\em Graduate Texts in Mathematics}.
\newblock Springer-Verlag, New York, 2002.

\bibitem{G85}
{\sc G{\"o}tze, F.}
\newblock Asymptotic expansions in functional limit theorems.
\newblock {\em J. Multivariate Anal. 16}, 1 (1985), 1--20.

\bibitem{Kant96}
{\sc Kantorovich, L.~V.}
\newblock {\em Selected works. {P}art {II}}, vol.~3 of {\em Classics of Soviet
  Mathematics}.
\newblock Gordon and Breach Publishers, Amsterdam, 1996.
\newblock Applied functional analysis. Approximation methods and computers,
  Translated from the Russian by A. B. Sossinskii, Edited by S. S. Kutateladze
  and J. V. Romanovsky.

\bibitem{Kar07}
{\sc Kargin, V.}
\newblock Berry-{E}sseen for free random variables.
\newblock {\em J. Theoret. Probab. 20}, 2 (2007), 381--395.

\bibitem{Kar13}
{\sc Kargin, V.}
\newblock An inequality for the distance between densities of free
  convolutions.
\newblock {\em Ann. Probab. 41}, 5 (2013), 3241--3260.

\bibitem{M92}
{\sc Maassen, H.}
\newblock Addition of freely independent random variables.
\newblock {\em J. Funct. Anal. 106}, 2 (1992), 409--438.

\bibitem{V85}
{\sc Voiculescu, D.}
\newblock Symmetries of some reduced free product {$C^\ast$}-algebras.
\newblock In {\em Operator algebras and their connections with topology and
  ergodic theory}, vol.~1132 of {\em Lecture Notes in Math.} Springer, Berlin,
  1985, pp.~556--588.

\bibitem{V86}
{\sc Voiculescu, D.}
\newblock Addition of certain noncommuting random variables.
\newblock {\em J. Funct. Anal. 66}, 3 (1986), 323--346.

\bibitem{W10}
{\sc Wang, J.-C.}
\newblock Local limit theorems in free probability theory.
\newblock {\em Ann. Probab. 38}, 4 (2010), 1492--1506.

\bibitem{Zor2}
{\sc Zorich, V.~A.}
\newblock {\em Mathematical analysis, {II}}.
\newblock Springer, Berlin, 2009.

\end{thebibliography}

\  \\ \ 
{\small \sc Faculty of Mathematics, Univ. Bielefeld, P.O.Box 100131, 33501 Bielefeld, Germany}\\
{\it E-mail address:} {\sf goetze@math.uni-bielefeld.de, areshete@math.uni-bielefeld.de}

\end{document}